\documentclass[a4paper, cleveref, autoref, thm-restate, numberwithinsect]{lipics-v2021}
\usepackage{mathtools}
\usepackage{physics}
\usepackage{graphicx}
\usepackage[dvipsnames, table]{xcolor}
\usepackage{cite}
\usepackage{bm}
\usepackage{booktabs}
\usepackage{diagbox}
\usepackage{MnSymbol}
\usepackage{complexity}
\usepackage{subcaption}
\usepackage{tabularx}
\usepackage{thm-restate}
\usepackage{xargs}
\usepackage{ifthen}

\usepackage[capitalise, noabbrev]{cleveref}

\newsavebox{\imagebox}

\newcolumntype{Y}{>{\centering\arraybackslash}X}
\newcommand{\ra}[1]{\renewcommand{\arraystretch}{#1}}
\newcommand{\shortref}[1]{{\color{black} Thm.\,\ref{#1}}}
\newcommand{\shortrefcor}[1]{{\color{black} Cor.\,\ref{#1}}}

\nolinenumbers
\hideLIPIcs

\definecolor{red30}{RGB}{229,242,251}
\definecolor{green30}{RGB}{132,195,237}
\definecolor{openColor}{RGB}{234,234,234}

\definecolor{blue1}{rgb}{0.274,0.392,0.666}
\definecolor{blue2}{rgb}{0.529,0.607,0.784}
\definecolor{blue3}{rgb}{0.784,0.819,0.901}
\definecolor{brown1}{rgb}{0.654,0.509,0.18}
\definecolor{brown2}{rgb}{0.776,0.682,0.466}
\definecolor{brown3}{rgb}{0.898,0.854,0.752}
\definecolor{cyan1}{rgb}{0.137,0.631,0.878}
\definecolor{cyan2}{rgb}{0.439,0.76,0.921}
\definecolor{cyan3}{rgb}{0.741,0.89,0.964}
\definecolor{gray1}{gray}{0.301}
\definecolor{gray2}{gray}{0.596}
\definecolor{gray3}{gray}{0.89}
\definecolor{green1}{rgb}{0,0.588,0.509}
\definecolor{green2}{rgb}{0.352,0.733,0.682}
\definecolor{green3}{rgb}{0.701,0.878,0.854}
\definecolor{maygreen1}{rgb}{0.549,0.713,0.235}
\definecolor{maygreen2}{rgb}{0.709,0.815,0.505}
\definecolor{maygreen3}{rgb}{0.866,0.913,0.772}
\definecolor{orange1}{rgb}{0.874,0.607,0.105}
\definecolor{orange2}{rgb}{0.921,0.745,0.419}
\definecolor{orange3}{rgb}{0.964,0.882,0.733}
\definecolor{purple1}{rgb}{0.639,0.062,0.486}
\definecolor{purple2}{rgb}{0.768,0.392,0.666}
\definecolor{purple3}{rgb}{0.894,0.717,0.847}
\definecolor{red0}{rgb}{1,0.192,0.075}
\definecolor{red1}{rgb}{0.635,0.133,0.137}
\definecolor{red2}{rgb}{0.764,0.439,0.439}
\definecolor{red3}{rgb}{0.89,0.741,0.741}
\definecolor{signalblue1}{rgb}{0.227,0.463,0.941}
\definecolor{signalblue2}{rgb}{0.42,0.562,0.941}
\definecolor{signalblue3}{rgb}{0.573,0.694,0.941}
\definecolor{signalblue4}{rgb}{0.749,0.812,0.941}
\definecolor{violet1}{rgb}{0.514,0.376,0.824}
\definecolor{violet2}{rgb}{0.639,0.569,0.824}
\definecolor{violet3}{rgb}{0.776,0.737,0.906}
\definecolor{yellow1}{rgb}{0.988,0.898,0}
\definecolor{yellow2}{rgb}{0.992,0.933,0.352}
\definecolor{yellow3}{rgb}{0.996,0.968,0.701}

\colorlet{colorNPK}{red2}
\colorlet{colorNP}{red3}
\colorlet{colorP}{signalblue3}
\definecolor{colorOpen}{RGB}{234,234,234}

\definecolor[named]{lipicsGray}{rgb}{0.31,0.31,0.33}

\usepackage{enumitem}
\setlist[enumerate]{leftmargin=*, align=left, nosep}
\setlist{nosep}
\newcommand{\enumstyle}[1]{\textcolor{lipicsGray}{\sffamily\bfseries\upshape\mathversion{bold}{#1}}}
\newenvironment{myenum}[1]
	{\begin{enumerate}[label=\enumstyle{({#1\arabic*})}, ref={#1}\arabic*]}
	{\end{enumerate}}

\newcommand{\N}{\mathbb{N}}
\newcommand{\setmid}{\ensuremath{\; \colon \;}}
\DeclarePairedDelimiterX\set[1]\lbrace\rbrace{\def\given{\setmid}#1}
\DeclarePairedDelimiter\ceil\lceil\rceil

\newcommand{\by}[2]{#1{\,:\,}#2}
\newcommand{\kneser}[2]{K_{#1{:}#2}}
\newcommand{\mycielski}{\operatorname{M}}

\newcommand{\nph}{$\NP$-hard}
\newcommand{\npc}{$\NP$-complete}

\newcommand{\chif}{\chi_{\lab{f}}}
\newcommand{\chib}{\chi}

\DeclareMathOperator{\indegree}{indeg}

\newcommand{\join}{\lor}
\newcommandx{\hit}[2]{%
   \ifthenelse{ \equal{#2}{} }
      {\ensuremath{\operatorname{hit}(#1)}}
      {\ensuremath{\operatorname{hit}(#1,#2)}}
}
\newcommand{\lab}[1]{\text{\fontfamily{iwona}\selectfont #1}}
\newcommandx{\cn}[3]{%
   \ifthenelse{ \equal{#3}{} }
      {\ensuremath{\operatorname{c}_{\text{\fontfamily{iwona}\selectfont #1}}^{#2}}}
      {\ensuremath{\operatorname{c}_{\text{\fontfamily{iwona}\selectfont #1}}^{#2}(#3)}}
}
\newcommandx{\cl}[1]{\ensuremath{\operatorname{c}_{\text{\fontfamily{iwona}\selectfont l}}^{#1}}}
\newcommandx{\cg}[1]{\ensuremath{\operatorname{c}_{\text{\fontfamily{iwona}\selectfont g}}^{#1}}}

\newcommand{\calG}{\mathcal{G}}
\newcommand{\calH}{\mathcal{H}}
\newcommand{\calF}{\mathcal{F}}
\newcommand{\calS}{\mathcal{S}}
\newcommand{\calB}{\mathcal{B}}
\newcommand{\calBc}{\mathcal{B}_{\mathrm{c}}}
\newcommand{\biclique}{\calBc}
\newcommand{\calI}{\mathcal{I}}
\newcommand{\calK}{\mathcal{K}}

\newcommand{\bip}{\mathcal{B}}
\newcommand{\stars}{\mathcal{S}}
\newcommand{\compl}{\mathcal{K}}

\newcommand{\etri}{\mathcal{E}_3} 
\newcommand{\ebound}[1]{\mathcal{E}_{#1}}

\newcommand{\probglob}[1]{\textsf{global-}{#1}\textsf{-covering} problem}
\newcommand{\probloc}[1]{\textsf{local-}{#1}\textsf{-covering} problem}
\newcommand{\probglobloc}[1]{\textsf{global-} and \textsf{local-}{#1}\textsf{-covering} problem}
\newcommand{\probpartglob}[1]{\textsf{partial-global-}{#1}\textsf{-covering} problem}
\newcommand{\probpartloc}[1]{\textsf{partial-local-}{#1}\textsf{-covering} problem}

\newcommand{\probdecomp}[1]{{#1}\textsf{-decomposition} problem}

\newcommand{\glob}[1]{\textsf{global-}{#1}\textsf{-covering}}
\newcommand{\loc}[1]{\textsf{local-}{#1}\textsf{-covering}}

\newcounter{counter1}
\newcounter{counter2}
\newcounter{counter3}
\newcounter{counter4}

\author{Miriam Goetze}{Karlsruhe Institute of Technology, Germany}{miriam.goetze@kit.edu}{https://orcid.org/0000-0001-8746-522X}{funded by the Deutsche Forschungsgemeinschaft (DFG, German Research Foundation) -- 520723789}

\author{Lucas Schwebler}{Karlsruhe Institute of Technology, Germany}{}{}{}

\authorrunning{M. Goetze, L. Schwebler}
\title{On the Computational Complexity of Local and Global Covering Numbers}

\newtheorem{question}[theorem]{Question}
\crefname{question}{Question}{Questions}
\Crefname{question}{Question}{Questions}
\crefname{conjecture}{Conjecture}{Conjectures}
\Crefname{conjecture}{Conjecture}{Conjectures}
\theoremstyle{definition}



\begin{document}

\maketitle

\begin{abstract}
    The global and local $\calG$-covering number~$\cg{\calG}(H)$ and~$\cl{\calG}(H)$ encode how well the edges of a graph~$H$ can be covered with graphs from a graph class~$\calG$: in the global setting, we minimize the number of graphs from~$\calG$ required, in the local setting how often a vertex is hit by the graphs of the cover.
    Within this work we consider for~$\calG$ the graph classes~$\calB$ of all bipartite and $\calBc$ of all complete bipartite graphs.
    We give a tight lower bound on $\cl{\calB}(H)$ in terms of the fractional chromatic number of~$H$, thereby giving a local analogue of a result by Harary, Hsu and Miller.
    Answering a question by Fishburn and Hammer, we show that it is $\NP$-hard to determine $\cl{\calBc}(H)$.
    Further, we provide a finite and  monotone graph class~$\calG$ such that $\cg{\calG}(H)$ can be computed in constant time for every graph~$H$ while determining~$\cl{\calG}(H)$ is $\NP$-hard.
    This yields a natural example to a question raised by Knauer and Ueckerdt.
\end{abstract}

\section{Introduction}

Covering the edges of a given graph~$H$ with graphs of a class~$\calG$ permits to decompose~$H$ into simpler pieces.
We call such a cover a~$\calG$-cover (see \cref{par:covers} for formal definitions and \cref{fig:examples_G-cover-k7} for an example).
In general, we are interested in minimizing the number of graphs from~$\calG$ in such a cover.
The smallest value that can be achieved is called the \emph{global $\calG$-covering number}~$\cn{g}{\calG}{H}$ of~$H$.
From an algorithmic point of view $\calG$-covers where $\calG$ corresponds to the class~$\calBc$ of all complete bipartite graphs are of interest, as such covers can be interpreted as a compression of~$G$.
Here, in particular when we are given a $\calBc$-cover of a graph~$H$ with small \emph{weighted size}  (that is the sum of vertices over all graphs of the cover) there are significant improvements upon the runtime for computing the vertex and the edge connectivity of general~$H$ and for computing a maximum matching when~$H$ is bipartite, see \cref{overview:algorithmic} for an overview of related results.

In a $\calBc$-cover~$\varphi\colon B_1 \cupdot \dots \cupdot B_t \to H$ of small weighted size of an $n$-vertex graph~$H$, a single vertex~$v \in V(H)$ may still lie in many of the graphs~$B_i$ of the cover. 
That is, adjacency testing for such compressions of~$H$ (e.g. \emph{adjacency labelings}) may require $\mathcal{O}(n\log(n))$ time \cite{cardinalImplicitRepresentationsPolynomial2026}.
Yet, if each vertex of~$H$ only belongs to at most~$t$ of the graphs~$B_i$, few bits are required to represent each vertex and adjacency can be tested in~$\mathcal{O}(t\log(n))$ time.
Note that here, the~$\calBc$-cover may consist of many graphs~$B_i$, we only need it to be locally small, i.e., that each vertex is contained in few graphs of the cover.
In general, the smallest~$t$ such that there exists a $\calG$-cover of~$H$ where each vertex is only contained in at most~$t$ graphs of the cover is called the \emph{local $\calG$-covering number}~$\cn{l}{\calG}{H}$, see \cref{fig:examples_G-cover-shift} for an example.

The aim of this paper is to bridge the gap between known results for computing global and local covering numbers.
Global $\calG$-covering numbers have been well-studied for various graph classes~$\calG$ (see \cite{schwartzOverviewGraphCovering2022} for an overview), such as
\begin{itemize}
    \item bipartite and complete bipartite graphs \cite{hararyBiparticityGraph1977,fishburn1996BipartiteDimension,pintoBicliqueCoversPartitions2014,wattsFractionalBicliqueCovers2006a,juknaCoveringGraphsComplete2009,tuzaCoveringGraphsComplete1984}
    \item complete graphs \cite{bollobasCliqueCoveringsEdges1993,brighamUpperBoundsEdge1984a,erdosRepresentationGraphSet1966} 
    \item forests, star forests, caterpillar forests and linear forests \cite{nash-william1964,knauer2016threeways,goncalvesCaterpillarArboricityPlanar2007,goncalvesStarCaterpillarArboricity2009}
    \item outerplanar and planar graphs \cite{mutzelThicknessGraphsSurvey1998,goncalvesEdgePartitionPlanar2005}
    \item interval graphs \cite{goncalvesCaterpillarArboricityPlanar2007,goncalvesStarCaterpillarArboricity2009,kostochkayEveryOuterplanarGraph1999} \footnote{Gon{\c c}alves and Orchem \cite{goncalvesCaterpillarArboricityPlanar2007,goncalvesStarCaterpillarArboricity2009} work on $\calF$-covers where $\calF$ is the class of all caterpillar forests. As the caterpillar forests are contained in the class~$\calI$ of interval graphs, the upper bounds on the global $\calF$-covering number translate to global $\calI$-covering numbers.}
\end{itemize}
For an overview of known results for computing global covering numbers see \cref{tab:overview}.
Yet, in comparison, little is known about bounds on local covering numbers and the computational complexity in the local setting, see \cref{overview:local} for an overview.

\subsection{Contribution}

In this work we continue the study of the complexity of computing local $\calG$-covering numbers, thereby answering open questions raised by Fishburn and Hammer \cite[p.\,148]{fishburn1996BipartiteDimension}, and Knauer and Ueckerdt \cite[Question~26]{knauer2016threeways}.

\begin{table}[ht]
	\ra{1.3}
	\def\smallDist{3}
	\def\largeDist{8}
	\small
	\centering
	\begin{subtable}[t]{\textwidth}
		\centering
		\begin{tabularx}{\textwidth}{c *{4}{Y}}
			\toprule
			\arrayrulecolor{white}
			& \multicolumn{1}{Y}{forests} & \multicolumn{1}{Y}{star forests} & \multicolumn{1}{Y}{caterpillar forests} & \multicolumn{1}{Y}{linear forests} \\
			global & \multicolumn{1}{|c|}{\cellcolor{colorP} \cite{gabowForestsFramesGames1992,gabowAlgorithmsGraphicPolymatroids1998}} &
			\multicolumn{1}{|c|}{\cellcolor{colorNPK} \cite{hakimiStarArboricityGraphs1996, goncalvesStarCaterpillarArboricity2009}} & \multicolumn{1}{|c|}{\cellcolor{colorNPK} \cite{goncalvesStarCaterpillarArboricity2009, schermerRectangleVisibilityGraphs1996}} & \multicolumn{1}{|c|}{\cellcolor{colorNPK} \cite{perocheNPcompletenessProblemsPartitioning1984}} \\
			\cmidrule{1-3}
			local &  \multicolumn{1}{|c|}{\cellcolor{colorP} \cite{knauer2016threeways}\footnotemark
\setcounter{counter1}{\value{footnote}}} &
			\multicolumn{1}{|c|}{\cellcolor{colorP} \cite{knauer2016threeways}} & \multicolumn{1}{|c|}{\cellcolor{colorOpen} ? \cite[p.\,756]{knauer2016threeways}} & \multicolumn{1}{|c|}{\cellcolor{colorNPK}\cite{stumpf2015CoveringNumbersDifferent}}\\
			\arrayrulecolor{black}\bottomrule

			\addlinespace
            \toprule
			\arrayrulecolor{white}
			& \multicolumn{1}{Y}{bounded stars $\stars_{\leq d}$ for $d \geq 3$} & \multicolumn{1}{Y}{$\{G\}$ for connected $r$-regular $G$, $r \geq 2$} & \multicolumn{1}{Y}{$\{K_{1,d}\}$ for $d \geq 3$} & \multicolumn{1}{Y}{at most three edges $\etri$} \\
			global & \multicolumn{1}{|c|}{\cellcolor{colorNP}\shortref{thm:fin-glob-maxe-np}} &  \multicolumn{1}{|c|}{\cellcolor{colorNP}\shortref{thm:fin-glob-maxe-np}} & \multicolumn{1}{|c|}{\cellcolor{colorNP}\shortref{thm:fin-glob-maxe-np}} & \multicolumn{1}{|c|}{\cellcolor{colorP}\shortref{thm:global-easy-local-hard}} \\
			\cmidrule{1-3}
			local & \multicolumn{1}{|c|}{\cellcolor{colorP}\shortref{thm:fin-local-bounded-star}} & \multicolumn{1}{|c|}{\cellcolor{colorNP}\cite{schweblerGraphCoveringAlgorithnms2026}} & \multicolumn{1}{|c|}{\cellcolor{colorNP}\cite{schweblerGraphCoveringAlgorithnms2026}} & \multicolumn{1}{|c|}{\cellcolor{colorNP}\shortref{thm:global-easy-local-hard}} \\
			\arrayrulecolor{black}\bottomrule

			\addlinespace
            \toprule
			\arrayrulecolor{white}
			& \multicolumn{1}{Y}{planar graphs} & \multicolumn{1}{Y}{outerplanar graphs} & \multicolumn{1}{Y}{cacti} & \multicolumn{1}{Y}{treewidth at most $t$ for $t \geq 2$} \\
			global & \multicolumn{1}{|c|}{\cellcolor{colorNPK}\cite{mansfieldDeterminingThicknessGraphs1983,lee2026determiningOuterthickness}{\protect\footnotemark\setcounter{counter4}{\value{footnote}}}} &  \multicolumn{1}{|c|}{\cellcolor{colorNPK}\cite{lee2026determiningOuterthickness}{\protect\footnotemark[\value{counter4}]}} & \multicolumn{1}{|c|}{\cellcolor{colorNPK}\cite{lee2026determiningOuterthickness}{\protect\footnotemark[\value{counter4}]}} & \multicolumn{1}{|c|}{\cellcolor{colorNPK} \cite{lee2026determiningOuterthickness}{\protect\footnotemark[\value{counter4}]}} \\
			\cmidrule{1-3}
			local & \multicolumn{1}{|c|}{\cellcolor{colorNPK} \cite{schweblerGraphCoveringAlgorithnms2026}{\protect\footnotemark[\value{counter4}]}} & \multicolumn{1}{|c|}{\cellcolor{colorNPK} \cite{schweblerGraphCoveringAlgorithnms2026}{\protect\footnotemark[\value{counter4}]}} & \multicolumn{1}{|c|}{\cellcolor{colorNPK}\cite{schweblerGraphCoveringAlgorithnms2026}{\protect\footnotemark[\value{counter4}]}} & \multicolumn{1}{|c|}{\cellcolor{colorNPK}\cite{schweblerGraphCoveringAlgorithnms2026}{\protect\footnotemark[\value{counter4}]}} \\
			\arrayrulecolor{black}\bottomrule

			\addlinespace
            \toprule
			\arrayrulecolor{white}
			& \multicolumn{1}{Y}{bipartite graphs} & \multicolumn{1}{Y}{bicliques $\calBc$} & \multicolumn{1}{Y}{cliques $\compl$} & interval graphs  \\
			global & \multicolumn{1}{|c|}{\cellcolor{colorNPK}\cite{hararyBiparticityGraph1977,karp1972reducibility}{\protect\footnotemark\setcounter{counter2}{\value{footnote}}}}& \multicolumn{1}{|c|}{\cellcolor{colorNP}\shortref{thm:global-clique-bc-np}, \cite{orlin1977contentment,mullerEdgePerfectnessClasses1996}} &
			\multicolumn{1}{|c|}{\cellcolor{colorNP}\shortref{thm:global-clique-bc-np}, \cite{orlin1977contentment,kouCoveringEdgesCliques1978}} &  \multicolumn{1}{|c|}{\cellcolor{colorNPK} \cite{jiangRecognizingDIntervalGraphs2013}} \\
			\cmidrule{1-3}
			local & \multicolumn{1}{|c|}{\cellcolor{colorNPK}\shortref{thm:local-bip-np}} & \multicolumn{1}{|c|}{\cellcolor{colorNP}\shortref{thm:local-bc-np}} &
			\multicolumn{1}{|c|}{\cellcolor{colorNPK}\shortrefcor{thm:local-clique-np}, \cite{poljak1981complexityRepresentation}{\protect\footnotemark\setcounter{counter3}{\value{footnote}}}} & \multicolumn{1}{|c|}{\cellcolor{colorNPK} \cite{stumpf2015CoveringNumbersDifferent}}  \\
			\arrayrulecolor{black}\bottomrule
		\end{tabularx}
		\medskip
		\caption{\;\; \textcolor{colorP}{$\blacksquare$} in \P \;\; \textcolor{colorNP}{$\blacksquare$} weakly $\NP$-complete\footnote{To our knowledge, it is not known whether these problems are also strongly $\NP$-complete.}\;\; \textcolor{colorNPK}{$\blacksquare$} strongly $\NP$-complete \;\; \textcolor{colorOpen}{$\blacksquare$} Unknown}
	\end{subtable}
	\caption{
		Overview of the main complexity results obtained in this paper.
		The cells correspond to the \glob{$\calG$} and \loc{$\calG$} problem where $\calG$ depends on the column.
		In these problems, $k$ is given in the input.
		If we show $\NP$-hardness for some fixed $k$, the cell is marked accordingly.
	}
	\label{tab:overview}
\end{table}
\footnotetext[\value{counter1}]{Knauer and Ueckerdt show that the local and global $\calF$-covering number for the class~$\calF$ of all forests coincide \cite[Theorem~8]{knauer2016threeways}. The local version can thus be reduced to the global variant.}
\footnotetext[\value{counter4}]{Lee, Liu and Tsai show that the $k$-\probglob{$\calG$} is $\NP$-complete in all these cases for all~$k \geq 3$ \cite{lee2026determiningOuterthickness} Adapting their technique to the local setting, Schwebler yields the same result in the local setting \cite{schweblerGraphCoveringAlgorithnms2026}.}
\footnotetext[\value{counter2}]{Harary, Hsu and Miller relate the global $\calB$-covering number of a graph~$H$ to its chromatic number~$\chi(H)$ (cf. \cref{lem:cover_with_log_chi_bipartite_graphs}). As determining the chromatic number is $\NP$-complete \cite{karp1972reducibility}, so is the \probglob{$\calB$}.}
\footnotetext[\value{counter3}]{We only prove that the \probloc{$\compl$} is $\NP$-complete. Poljak, Rödl and Turzík show the stronger result that the $k$-\probloc{$\compl$} is $\NP$-complete for every fixed $k \geq 4$, while it can be solved in polynomial time for~$k \leq 2$ \cite{poljak1981complexityRepresentation}.}

\subparagraph*{A lower bound on local bipartite covering numbers.}

For the class~$\calB$ of all bipartite graphs the global $\calB$-covering number~$\cn{g}{\calB}{}$ is well-understood: it only depends on the chromatic number of the host graph~$H$.
\begin{theorem}[{Harary, Hsu, Miller~\cite{hararyBiparticityGraph1977}}] 
    \label{lem:cover_with_log_chi_bipartite_graphs}
    Let~$\calB$ be the class of all bipartite graphs.
    For every graph~$H$, we have $\cn{g}{\calB}{H} = \big\lceil\log(\chi(H))\big\rceil$.
\end{theorem}
For local $\calB$-covering numbers no similar relationship is known. 
In fact, the local $\calB$-covering number and the chromatic number can be arbitrarily far apart \cite[Proposition~21]{goetze2025boundedness}.
We present a lower bound on the local $\calB$-covering number in what is known as the \emph{fractional chromatic number}~$\chif$ (see \cref{par:fractional_number} for a definition).
\begin{restatable}{theorem}{fractionalChromaticCoverBip}
	\label{thm:fractional-chromatic-cover-bip}
	For the class~$\calB$ of all bipartite graphs we have $\cn{l}{\calB}{H} \geq \ceil{\log_2(\chif(H))}$ for every graph~$H$.
\end{restatable}

From \cref{lem:cover_with_log_chi_bipartite_graphs} we obtain $\cn{g}{\calBc}{K_n} = \ceil{\log(n)}$ for every complete graph~$K_n$ as we may assume that all guests of a $\calB$-cover of~$K_n$ are complete bipartite graphs. 
Fishburn and Hammer ask whether we also have $\cn{l}{\calBc}{K_n} = \ceil{\log(n)}$ \cite[p.\,147]{fishburn1996BipartiteDimension}. 
Yet, even before they raised the question this has been known to be true due to results by Hansel \cite{hansel1964NombreMinimal} and Krichevskii \cite{Krichevskii1963ComplexityOC}.
Later on, more proofs using different techniques have been obtained:
a proof due to Radhakrishnan using the probabilistic method \cite{radhakrishnanen2001tropy} is presented in {\cite[Lemma 3.7]{jukna2013Complexity}}.\footnote{Jukna and Sergeev give a lower bound on the weighted size~$\sum_{v \in V(H)} \abs{\varphi^{-1}(v)}$ of a $\calBc$-cover~$\varphi$ for any graph~$H$.} 
Dong and Liu give an explicit proof \cite[Theorem~3.1]{dong2017decompositionIntoCompleteBipartite} which characterizes the structure of optimal local $\calBc$-covers of~$K_n$.\footnote{They give a full proof for $\calBc$-decompositions of~$K_n$ and note how the proof can be altered in order to obtain the corresponding result for $\calBc$-covers.} 
\begin{restatable}{theorem}{DongLiuClCgCompbipKn}
	\label{thm:local-compbip-kn}
	For the class $\calBc$ of all complete bipartite graphs, we have $\cn{l}{\calBc}{K_n} = \ceil{\log_2(n)}$.
\end{restatable}
With \cref{thm:fractional-chromatic-cover-bip}, we easily derive an alternative proof of \cref{thm:local-compbip-kn}.

\subparagraph*{Complexity of the local \texorpdfstring{$\bm{\calB}$}{B}- and \texorpdfstring{$\bm{\calBc}$}{Bc}-covering problem.}

As the global $\calB$-covering number $\cg{\calB}(H)$ for the class~$\calB$ of all bipartite graphs is determined by the chromatic number (cf. \cref{lem:cover_with_log_chi_bipartite_graphs}), the \probglob{$\calB$} is $\NP$-complete \cite{karp1972reducibility}.
Yet, while \cref{thm:fractional-chromatic-cover-bip} provides a tight lower bound on~$\cl{\calB}(H)$ in the fractional chromatic number~$\chif(H)$ of~$H$, the two parameters~$\cl{\calB}(H)$ and~$\chif(H)$ are not related by a function (see \cref{thm:fractional-chromatic-cover-bip-not-tight}).
Thus, even though determining whether a graph~$H$ has fractional chromatic number at most~$k$ is $\NP$-hard \cite{grotschelEllipsoidMethodIts1981}, this does not immediately yield $\NP$-hardness of the \probloc{$\calB$}.
We show $\NP$-completeness via a reduction from the $3$-\textsf{vertex}-\textsf{coloring} problem (that is, decide whether~$\chi(H) \leq 3$ for a given graph~$H$).
\begin{restatable}{theorem}{localBcoveringNPC}
	\label{thm:local-bip-np}
	The $2$-\probloc{$\bip$} is $\NP$-complete.
\end{restatable}

Fishburn and Hammer ask whether the \probglobloc{$\calBc$} are $\NP$-complete for the class~$\calBc$ of all complete bipartite graphs \cite[p.\,148]{fishburn1996BipartiteDimension}.
This has been shown to be the case in the global setting by Orlin, and independently by Müller \cite{mullerEdgePerfectnessClasses1996,orlin1977contentment}: Orlin proves that the \probglob{$\calBc$} is $\NP$-complete for bipartite hosts, Müller that this remains true even when the hosts are \emph{chordal bipartite}, i.e, are bipartite graphs where every cycle of size at least~$6$ has a chord.
We reformulate Orlin's proof for the $\NP$-completeness of the global $\calBc$- and $\calK$-covering problem (here $\calK$ denotes the class of all complete graphs) within a framework which we can then extend to show $\NP$-completeness for the local $\calBc$- and $\calK$-covering problem.
As in the work of Orlin, we prove $\NP$-completeness via the decision problem for \emph{$S$-partial $\calG$-covers}.
An $S$-partial $\calG$-cover of a graph~$H$ and a subset~$S \subseteq E(H)$ of its edges consists of subgraphs~$G_1, \dots, G_t \in \mathcal{G}$ of~$H$ whose union covers all edges of~$S$. 
Using the global and the local $S$-partial $\calG$-cover as a stepping stone, we answer Fishburn and Hammer's question on the complexity of the \probloc{$\calBc$} and obtain alternative proofs for three other settings\footnote{For the class~$\calK$ of all complete graphs and the class~$\calBc$ of all complete bipartite graphs, the global \cite[Theorem~8.1]{orlin1977contentment} \cite[Proposition~3]{kouCoveringEdgesCliques1978} and local $\calK$-covering problem \cite[Theorem~2.1]{poljak1981complexityRepresentation} and the global $\calBc$-covering problem \cite[Theorem~6]{mullerEdgePerfectnessClasses1996}\cite[Theorem~8.1]{orlin1977contentment} have been known to be $\NP$-complete.}:
\begin{restatable}{theorem}{thmGlobLocCompbipBip}
	\label{thm:glob-loc-compbip-bip}
	The \probglobloc{$\calG$} are $\NP$-complete for the class~$\calG$ of
    \begin{itemize}
        \item all complete graphs,
		\item all complete bipartite graphs.
	\end{itemize}
\end{restatable}

\subparagraph*{Comparing the computational complexity of global and local covering problems.}
In many cases, the \probglob{$\calG$} is $\NP$-complete.
Yet, when the local covering number~$\cn{l}{\calG}{}$ can be computed in polynomial time and the global covering number~$\cn{g}{\calG}{}$ can be bounded in terms of~$\cn{l}{\calG}{}$ for a host class~$\calH$ (see \cite{goetze2025boundedness} for structural properties of~$\calG$ and~$\calH$ which yield what is known as $(\cn{g}{\calG}{}$,$\cn{l}{\calG}{}$)-\emph{boundedness}) we obtain an approximation algorithm for~$\cn{g}{\calG}{}$.

In general it seems that computing global $\calG$-covering numbers is harder than for local $\calG$-covering numbers, as observed in \cite[p.\,1314]{blasiusLocalUnionBoxicity2018}.
Indeed, for many graph classes~$\calG$ this seems to be the case  (cf. \cref{tab:overview}).
While computing the global covering number with star forests is known to be $\NP$-complete \cite[Theorem~4]{hakimiStarArboricityGraphs1996} \cite[Theorem~2]{goncalvesStarCaterpillarArboricity2009}, Knauer and Ueckerdt show that the local counterpart can be computed in polynomial time \cite[Theorem~25]{knauer2016threeways}.
For covers with cliques, both the global and the local covering problem turn out to be $\NP$-complete \cite[Theorem~8.1]{orlin1977contentment} \cite[Theorem~2.1]{poljak1981complexityRepresentation}.
The same holds for covers with the class~$\calI$ of interval graphs.
The \probglob{$\calI$} has been shown to be $\NP$-complete by Jiang \cite[Theorem~1]{jiangRecognizingDIntervalGraphs2013}.
Stumpf adapted the proof to obtain $\NP$-completeness for the local variant \cite[Section~6.1]{stumpf2015CoveringNumbersDifferent}.

Knauer and Ueckerdt ask in \cite[Question~26]{knauer2016threeways} whether there exists a host class~$\calH$ and a guest class~$\calG$ such that computing the local $\calG$-covering number is $\NP$-hard while the global $\calG$-covering number can be computed in polynomial time for all graphs in~$\calH$.
This has been answered in the affirmative by Stumpf.
They provide a general template for such examples by restricting the host class~$\calH$ \cite[Section~6.2]{stumpf2015CoveringNumbersDifferent}.
In fact, they also give examples where the global $\calG$-covering number~$\cn{g}{\calG}{H}$ can be computed in constant time for all~$H \in \calH$, while it is undecidable whether $\cn{l}{\calG}{H} \leq k$ \cite{stumpf2017ma}.
Yet, the polynomial-time algorithms for computing~$\cn{g}{\calG}{}$ heavily rely on the very specifically chosen class~$\calH$ of host graphs and cannot be lifted to general graphs~$H$.
We give an example of a monotone, finite graph class~$\calG$ such that $\cn{g}{\calG}{H}$ can be computed in constant time for every graph~$H$, while the \probloc{$\calG$} is $\NP$-complete.
\begin{restatable}{theorem}{thmGlobalEasyLocalHard}
	\label{thm:global-easy-local-hard}
	For the class $\etri$ of all graphs on at most three edges, we have:
	\begin{enumerate}[label=\enumstyle{(\roman*)}, ref=\roman*]
		\item\label{itm:global-easy} $\cn{g}{\etri}{H} = \ceil{\frac{\abs{E(H)}}{3}}$ for all host graphs $H$.
		\item\label{itm:local-hard} The \probloc{$\etri$} is $\NP$-complete.
	\end{enumerate}
\end{restatable}
This provides the first natural graph class with this behavior and a more satisfactory answer to the question of Knauer and Ueckerdt.

\subsection{Related Work}
\subparagraph*{Algorithmic aspects of \texorpdfstring{$\bm{\calBc}$}{Bc}-covers in the literature.}
\label{overview:algorithmic}
A dense $n$-vertex graph~$H$ with a $k$-global $\calBc$-cover~$\varphi$ is structurally similar to a graph with~$\mathcal{O}(kn)$ edges. 
Indeed, for every complete bipartite graph~$K_{s,t}$ of the cover~$\varphi$, we can replace its~$s \cdot t$ edges in~$H$ with $s+t$~edges by inserting a new vertex~$w$ connected to all endpoints of the copy of~$K_{s,t}$ in~$H$.
This property can be algorithmically exploited:
Given a $\calBc$-decomposition (that is a $\calBc$-cover with edge-disjoint graphs) of an $n$-vertex graph~$H$ with weighted size~$s$, Feder and Motwani give algorithms for computing matchings of maximum size in bipartite graphs, vertex and edge connectivity and all-pair-shortest-paths which only depend on~$n$ and~$s$ \cite{federCliquePartitionsGraph1995}.
In particular for dense graphs, this yields improvements upon the runtime as every graph admits a $\calBc$-cover of weighted size~$\mathcal{O}(\frac{n^2}{\log(n)})$ \cite{chungDecompositionGraphsComplete1983,tuzaCoveringGraphsComplete1984}.
The runtime for computing maximum matchings in bipartite graphs has subsequently been improved by Cabello, Cheng, Cheong and Knauer:
In general, such a matching can be computed in time $\mathcal{O}(\abs{E(H)}^{1+\varepsilon})$ for every~$\varepsilon > 0$ with high probability, using maximum-flow \cite{chenMaximumFlowMinimumCost2022}.
Yet, the runtime can be reduced to $\mathcal{O}(s^{1+\varepsilon})$ if we are given a compression based on a $\calBc$-cover of~$H$ of weighted size~$s$ \cite[p.\,13]{cabelloGeometricMatchingBottleneck2024}.
Bounds on the weighted size of $\calBc$-covers have been obtained for many graph classes based on intersection families and visibility graphs (see \cite{cardinalCompactRepresentationSemilinear2025} for an overview).
In fact, lower bounds on the weighted size of~$\calBc$-covers of a point-hyperplane incidence graph yield lower bounds on the runtime of partitioning algorithms counting the number of point-hyperplane incidences \cite{ericksonNewLowerBounds1996}.

\subparagraph*{Local covering numbers in the literature.}
\label{overview:local}
Fishburn and Hammer consider global and local $\calBc$-covering numbers \cite{fishburn1996BipartiteDimension} for the class~$\calBc$ of complete bipartite graphs.
While their question on the complexity of computing global $\calBc$-covering numbers has been answered \cite{mullerEdgePerfectnessClasses1996,orlin1977contentment}, the complexity of the \probloc{$\calBc$} remained open.
We answer this question within this work (cf. \cref{thm:glob-loc-compbip-bip}).
In \cite{javadiLocalCliqueCovering2012} the local covering numbers with complete graphs are considered, which had been introduced in \cite{skumsEdgeIntersectionGraphs2009} (Skums, Suzdal and Tyshkevich mostly consider $k$-local decompositions into complete graphs).
Pinto compares \emph{decomposition numbers}\footnote{They use the term \emph{partition number} instead. Yet, within this work we use decomposition when considering edge-partitions, as partition is mostly used in the literature for vertex-partitions.} (that is the minimum number of edge-disjoint graphs from~$\calG$ required to cover a graph~$H$) to covering numbers. 
They observe that the global $\calBc$-partition number can be bounded in terms of the global $\calBc$-covering number, yet the analogous statement in the local setting does not hold \cite{pintoBicliqueCoversPartitions2014}.
Recently, the interest in local $\calBc$-covering numbers has increased due to its application in adjacency labeling. 
Cardinal and Yuditsky give upper bounds on the local $\calBc$-covering number~$\cn{l}{\calBc}{H}$ of graphs~$H \in \calH$ in the number~$\abs{V(H)}$ of vertices for many classes~$\calH$ \cite{cardinalCompactRepresentationSemilinear2025}.
In particular, they prove a logarithmic upper bound on~$\cn{l}{\calBc}{}$ for what is known as \emph{semilinear families} (which include interval graphs, permutation graphs, circle graphs, box intersection graphs and intersection graphs of rectilinear objects), thereby unifying proofs of known bounds on the size of adjacency labelings for different families~$\calH$. 
Analyzing known algorithms for computing $\calBc$-decompositions, they compute small adjacency labelings of what is known as capped graphs \cite[p.\,22]{cardinalImplicitRepresentationsPolynomial2026}.

Knauer and Ueckerdt give sharp bounds on global and local $\cal{G}$-covering numbers for specific host classes (such as planar and outerplanar graphs, and graphs of bounded treewidth) where $\calG$ is the class of star, caterpillar and linear forests and interval graphs \cite{knauer2016threeways}.
Apart from the global $\calI$-covering number~$\cn{g}{\calI}{}$ for the class~$\calI$ of all interval graphs (known as the \emph{track number}), another variant called the \emph{interval number}~$\cn{f}{\calI}{}$ received much attention \cite{trotterDoubleMultipleInterval1979,scheinermanIntervalNumberPlanar1983,gueganIntervalNumberPlanar2021,hopkinsBoundIntervalNumber1981, hopkinsIntervalNumberComplete1984,erdosNoteIntervalNumber1985,scheinermanMaximumIntervalNumber1987,andreaeExtremalProblemConcerning1986,baloghIntervalNumberSpecial2004,griggsExtremalValuesInterval1980,scheinermanIntervalNumberChordal1988,andreaeIntervalNumberTriangulated1987}, which is closely related to the local $\calI$-covering number~$\cn{l}{\calI}{}$ introduced in \cite{knauer2016threeways}.

\section{Preliminaries}
\label{sec:preliminaries}

\subparagraph*{\texorpdfstring{$\bm{\calG}$}{G}-covers and decompositions.}
\label{par:covers}
For a graph class~$\calG$ (called the \emph{guest class}) we say that subgraphs~$G_1, \dots, G_t \in \calG$ of a graph~$H$ form a \emph{$\calG$-cover of~$H$} if their union covers all edges of~$H$.
That is, there is an injective\footnote{We say that $\varphi$ is injective if the restriction~$\varphi_{\mid G_i}$ of $\varphi$ to each graph~$G_i$ is injective.}, edge-surjective graph homomorphism~$\varphi\colon G_1 \cupdot \dots \cupdot G_t \to H$ where $G_i \in \calG$ for all~$i \in [t]$.\footnote{In general a $\calG$-cover~$\varphi$ may correspond to a not necessarily injective graph homomorphism \cite{knauer2016threeways,goetze2025boundedness}, that is the graphs~$G_i$ may not be subgraphs of~$H$. Yet all graph homomorphisms we consider in this work are injective.}
If the subgraphs~$\varphi(G_i)$ are all edge-disjoint, $\varphi$ is a \emph{$\calG$-decomposition}, see \cref{fig:example_G-cover} for examples.
We write~$G$-decomposition when~$\calG$ corresponds to a single graph~$G$.
We call~$H$ the \emph{host} and the graphs~$G_i$ the \emph{guests}.
An edge~$e \in E(H)$ is \emph{covered} by a guest~$G_i$ if~$e \in \varphi(E(G_i))$ and a vertex $v \in V(H)$ is \emph{hit} by~$G_i$ if $v \in \varphi(V(G_i))$.
The \emph{hitcount} $\hit{\varphi}{v} = \abs{\varphi^{-1}(v)}$ of a vertex~$v \in V(H)$ corresponds to the number of guests~$G_i$ that contain~$v$.
When~$\varphi$ is clear from context, we may also write~$\hit{v}{}$.
\begin{figure}
    \centering
    \begin{subfigure}[t]{0.5\textwidth}
       \centering
        \includegraphics[page=6]{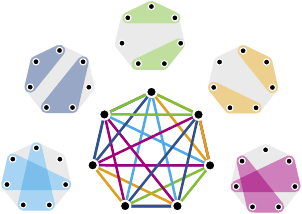}
        \caption{}
        \label{fig:examples_G-cover-k7}
    \end{subfigure}\hfill
    \begin{subfigure}[t]{0.5\textwidth}
        \centering
        \includegraphics[page=8]{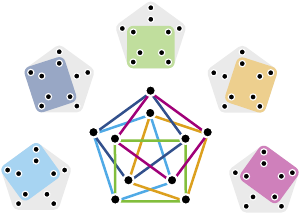}
        \caption{}
        \label{fig:examples_G-cover-shift}
    \end{subfigure}
    \caption{Examples of $\calG$-covers. (\subref{fig:examples_G-cover-k7}) A $10$-global $5$-local $\set{K_3}$-cover of~$K_7$ (taken from \cite[Figure~1]{goetze2025boundedness}). (\subref{fig:examples_G-cover-shift}) A $5$-global $2$-local $\calBc$-cover of a shift graph~$S$ of an orientation of~$K_5$ certifying $\cn{l}{\calBc}{S} \leq 2$.}
    \label{fig:example_G-cover}
\end{figure}

A cover~$\varphi\colon G_1 \cupdot \dots \cupdot G_t \to H$ is 
\begin{itemize}
\item \emph{$t$-global} if it consists of up to~$t$ guests, and 
\item \emph{$k$-local} if the hitcount of no vertex exceeds~$k$, that is~$\hit{\varphi}{v} \leq k$ for all $v \in V(H)$,
\end{itemize}
see \cref{fig:example_G-cover} for examples.
The smallest~$t$ such that there exists a $t$-global $\calG$-cover of~$H$ is the \emph{global $\calG$-covering number}~$\cn{g}{\calG}{H}$.
The \emph{local $\calG$-covering}~$\cn{l}{\calG}{H}$ is the smallest~$k$ such that there exists a $k$-local $\calG$-cover of~$H$.
If there exists no such cover, we set $\cn{g}{\calG}{H} = \infty$ and~$\cn{l}{\calG}{H} = \infty$ respectively.
Note that if~$K_2 \in \calG$, then the local and the global $\calG$-covering number are finite for every graph~$H$.
Every $t$-global $\calG$-cover of a graph~$H$ is in particular $t$-local.
We thus obtain the following.
\begin{lemma}[{Knauer and Ueckerdt \cite[Proposition~4]{knauer2016threeways}}]
    For every graph class~$\calG$ and every graph~$H$, we have 
    $\cn{l}{\calG}{H} \leq \cn{g}{\calG}{H}$.
\end{lemma}
Yet, $\cn{l}{\calG}{H}$ and~$\cn{g}{\calG}{H}$ can be arbitrarily far apart \cite{goetze2025boundedness}.

Note that for each~$G_i$ of a $t$-global $\calG$-cover $\varphi\colon G_1 \cupdot \dots \cupdot G_t \to H$ of a graph~$H$ with minimum $t$, there exists an edge~$e_i \in E(H)$ that is not covered by any other graph~$G_j$. 
That is, the global $\calG$-covering number of~$H$ (if finite) does not exceed the number of edges.
\begin{lemma}
	\label{thm:covering-small-witness}
    If a graph~$H$ admits a $\calG$-cover for some graph class~$\calG$, then 
    \[\cn{l}{\calG}{H} \leq \cn{g}{\calG}{H} \leq \abs{E(H)}.\]
\end{lemma}

\subparagraph*{\texorpdfstring{$\bm{S}$}{S}-partial \texorpdfstring{$\bm{\calG}$}{G}-covers.}
\label{par:S-partial-covers}
To prove $\NP$-hardness of several decision problems related to covering numbers, Orlin introduces the notion of \emph{$S$-partial $\calG$-covers} \cite{orlin1977contentment}. 
While a $\calG$-cover of a graph~$H$ needs to cover all edges of~$H$, an $S$-partial $\calG$-cover is only required to cover all edges~$S \subseteq E(H)$:
For a subset~$S \subseteq E(H)$ of the edges of a graph~$H$, we say that a set of subgraphs~$G_1, \dots, G_t \in \calG$ of~$H$ forms an an \emph{$S$-partial $\calG$-cover} if their union covers all edges of~$S$.
More formally, there is a (not necessarily edge-surjective) injective graph homomorphism~$\varphi\colon G_1 \cupdot \dots G_t \to H$ that hits each edge of~$S$.
We call edges $e \in E(H) - S$ \emph{optional}.
As for $\calG$-covers, we say that~$\varphi$ is 
\begin{itemize}
\item \emph{$t$-global} if it consists of $t$~subgraphs~$G_i$, and 
\item \emph{$\ell$-local} if~$\abs{\varphi^{-1}(v)} \leq \ell$ for every $v \in V(H)$.
\end{itemize}
When~$\calG$ is monotone, $\varphi$ can be restricted to a $\calG$-cover of~$S$.
Yet, in general $H$ might admit a $t$-global $S$-partial $\calG$-cover, but there might be no $t$-global $\calG$-cover of the graph induced by the edges of~$S$ for some~$S \subseteq E(H)$. 
Indeed, for the class~$\calBc$ of all complete graphs, the graph~$K_{n,n}$ admits a $1$-global $\calBc$-cover~$\varphi$, which is in particular an $S$-partial $\calBc$-cover for every~$S \subseteq E(K_{n,n})$.
However, when~$S \subseteq E(K_{n,n})$ does not correspond to a complete bipartite graph, there is no $1$-global $\calBc$-cover of~$S$.

\subparagraph*{\texorpdfstring{$\bm{\calG}$}{G}-covering problems.}
\label{par:covering_problems}
With regard to algorithmic complexity we consider two decision problems for a fixed graph class~$\calG$:
\begin{itemize}
    \item \probglob{$\calG$}: Given a graph $H$ and $k \in \mathbb{N}$, does $H$ admit a $k$-global $\calG$-cover?
    \item \probloc{$\calG$}: Given a graph $H$ and $k \in \mathbb{N}$, does $H$ admit a $k$-local $\calG$-cover?
\end{itemize}
Here, both~$H$ and~$k$ are part of the input, while the graph class~$\calG$ is fixed.
For fixed $k \in \N$ and a fixed graph class~$\calG$, we obtain two related decision problems:
\begin{itemize}
    \item $k$-\probglob{$\calG$}: Given a host graph $H$, does $H$ admit a $k$-global $\calG$-cover?
    \item $k$-\probloc{$\calG$}: Given a host graph $H$, does $H$ admit a $k$-local $\calG$-cover?
\end{itemize}
Clearly, $\NP$-hardness of the $k$-\probglob{$\calG$} implies $\NP$-hardness of $k$-\probloc{$\calG$}, and the same holds for the local setting. 
We say that the global covering problem for the graph class~$\calG$ is 
\emph{weakly $\NP$-complete} when the \probglob{$\calG$} is $\NP$-complete, and
\emph{strongly $\NP$-complete} when the $k$-\probglob{$\cal{G}$} is $\NP$-complete for some~$k$.
We use these terms similarly in the local setting.
See \cref{tab:overview} for an overview on the complexity of covering problems.

Similarly, by considering $S$-partial $\calG$-covers of a graph~$H$ instead of~$\calG$-covers of~$H$, we obtain two decision problems, in each of which the subset~$S \subseteq E(H)$, the graph~$H$ and $k \in \N$ are part of the input: the \probpartglob{$\calG$} and the \probpartloc{$\calG$}.

As the global $\calG$-covering number of every graph~$H$ does not exceed the number of edges if~$\cn{g}{\calG}{H}$ is finite (cf \cref{thm:covering-small-witness}), each graph~$H$ admits a small witness (if any). 
The same holds for $S$-partial $\calG$-covers of~$H$ as every $\calG$-cover is also an $S$-partial cover for every~$S \subseteq E(H)$.
In particular, we obtain the following:
\begin{observation}
\label{obs:NP-membership}
    If it can be verified in polynomial time whether a graph~$G$ belongs the a graph class~$\calG$, then the (\textsf{partial}-) \textsf{global}- and (\textsf{partial}-) \textsf{local}-$\calG$-\textsf{covering} problem lie in~$\NP$.
\end{observation}
It thus suffices to show $\NP$-hardness to obtain $\NP$-completeness.

\subparagraph*{Fractional chromatic number.}
\label{par:fractional_number}
The fractional chromatic number is a well-established parameter in the field of fractional graph theory, see \cite{scheinermanFractionalGraphTheory2011} for an introduction.
It is defined via \emph{$\by{a}{b}$-colorings}.
An \emph{$\by{a}{b}$-coloring} of a graph~$H$ is a vertex-coloring~$\Phi\colon V(H) \to \binom{[a]}{b}$ where each vertex is assigned a set of $b$ out of~$a$ possible colors such that~$\Phi(u) \cap \Phi(v) = \varnothing$ for every edge~$uv \in E(H)$.
That is, the color sets of adjacent vertices are disjoint.
We say that~$H$ is \emph{$\by{a}{b}$-colorable}, see \cref{fig:petersen} for an example.
\begin{figure}
    \centering
    \begin{subfigure}[t]{0.4\textwidth}
        \centering
        \includegraphics[page=1]{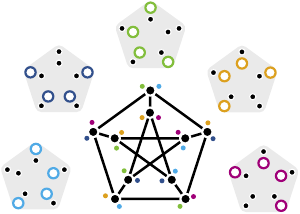}
        \caption{}
        \label{fig:petersen}
    \end{subfigure}\hfill
    \begin{subfigure}[t]{0.3\textwidth}
        \centering
        \includegraphics[page=2]{figures/examples_fractional_chrom.pdf}
        \caption{}
        \label{fig:matching}
    \end{subfigure}\hfill
    \begin{subfigure}[t]{0.3\textwidth}
        \centering
        \includegraphics[page=3]{figures/examples_fractional_chrom.pdf}
        \caption{}
        \label{fig:fractional_coloring}
    \end{subfigure}\hfill
    \caption{(\subref{fig:petersen}) A $\by{5}{2}$-coloring of the Petersen graph~$\kneser{5}{2}$ certifying~$\chif(\kneser{5}{2}) \leq \frac{5}{2}$. In fact, $\chif(\kneser{5}{2}) = \frac{5}{2}$ by \cref{thm:fract-chrom-ind}. (\subref{fig:matching}) A $\by{2}{1}$-coloring of a matching~$M$. (\subref{fig:fractional_coloring}) A graph~$H$ that can be covered with the Petersen graph~$\kneser{5}{2}$ (solid lines) and the matching~$M$ (dotted lines). The colors of three vertices are illustrated for the coloring~$\Phi$ obtained from the $\by{5}{2}$-coloring of~$\kneser{5}{2}$ and the $\by{2}{1}$-coloring of~$M$ as defined in the proof of \cref{thm:fractional-chromatic-cover}.}
    \label{fig:example_fractional_chrom}
\end{figure}
The minimum~$a$ such that~$H$ is~$\by{a}{b}$-colorable is the \emph{$b$-fold chromatic number}~$\chib_b(H)$.
The \emph{fractional chromatic number} is now defined as~$\chif(H) = \inf_{b \in \N} \frac{\chib_b(H)}{b}$.
Clearly, we have $\chif(H) \leq \chi_1(H) = \chi(H)$ and it is easy to see that $\chif$ is a monotone graph parameter:
\begin{observation}
\label{obs:fract_chi_monotone}
    For every graph~$H$, we have $\chif(H) \leq \chi(H)$ and if~$H'$ is a subgraph of a graph~$H$, then $\chif(H') \leq \chif(H)$.
\end{observation}
The fractional chromatic number is known to be always rational \cite[p.\,30]{scheinermanFractionalGraphTheory2011}.
\begin{lemma}
	\label{thm:fract-chrom-rational}
	For every graph $H$, there are integers $a$ and $b$ such that $H$ is $a{:}b$-colorable and $\chif(H) = \frac{a}{b}$.
	In particular, $\chif(H)$ is a rational number.
\end{lemma}
As every color class of a proper coloring of a graph~$H$ forms an independent set, we obtain a lower bound on the chromatic number of~$H$ in terms of the size~$\alpha(H)$ of a largest independent set.
In fact, this lower bound also applies to the fractional chromatic number \cite[Proposition~3.1.1]{scheinermanFractionalGraphTheory2011}.
\begin{lemma}
	\label{thm:fract-chrom-ind}
	For every graph $H$ we have $\frac{\abs{V(H)}}{\alpha(H)} \leq \chif(H)$.
\end{lemma}
\begin{proof}
	Consider an $\by{a}{b}$-coloring $\Phi\colon V(H) \to \binom{[a]}{b}$ of $H$.
	For each color $k \in [a]$, we denote by ${A(k) = \{v \in V(H) \mid k \in \Phi(v)\}}$ its color class, i.e., the vertices having color~$k$.
	Since adjacent vertices have disjoint color sets, $A(k)$ is an independent set and thus $\abs{A(k)} \leq \alpha(H)$.
	Note that $b \cdot \abs{V(H)} = \sum_{k \in [a]} \abs{A(k)} \leq a \cdot \alpha(H)$.
	Rearranging yields $\frac{\abs{V(H)}}{\alpha(H)} \leq \frac{a}{b}$.
	The bound now follows as $\chif(H)$ is the infimum of $\frac{a}{b}$ over all $\by{a}{b}$-colorings of $H$.
\end{proof}
In a proper coloring, each vertex of a clique receives a distinct color. 
That is, the size~$\omega(H)$ of a largest clique in~$H$ is a lower bound on the chromatic number.
Similarly, it provides a lower bound on the fractional chromatic number \cite{scheinermanFractionalGraphTheory2011}.
\begin{corollary}
	\label{thm:fract-chrom-clique}
	For every graph $H$ we have $\omega(H) \leq \chif(H)$.
\end{corollary}
\begin{proof}
    As the largest independent set of the complete graph~$K_n$ has size~$1$, \cref{thm:fract-chrom-ind} yields~$n \leq \chif(K_n)$.
	Since $\chif$ is monotone, we obtain $\omega(H) \leq \chif(H)$ for every graph $H$.
\end{proof}

Mycielski constructed triangle-free graphs with arbitrarily large chromatic number by providing a construction rule which augments the chromatic number by one in each step \cite{mycielski1955coloriage}.
Given a graph~$G$ we define~$\mycielski(G)$ as the graph obtained from~$G$ by 
\begin{itemize}
\item adding a copy~$v'$ of each vertex~$v$ and connecting~$v'$ to all neighbors of~$v$, and
\item adding a vertex~$x$ connected to all copies~$v'$, 
\end{itemize}
see \cref{fig:example_mycielski} for an example.
Mycielski showed that~$\chi(\mycielski(G)) = \chi(G)+1$ and~$\omega(\mycielski(G)) = \omega(G)$, thereby giving examples of graphs with arbitrarily large chromatic number that contain no large cliques.
In fact, their fractional chromatic number is also unbounded \cite[Theorem~3.3.4]{scheinermanFractionalGraphTheory2011}.
\begin{lemma}
\label{lem:mycielski_fractional_chrom}
    If~$G$ is a graph on at least one edge, then $\chif(\mycielski(G)) = \chif(G) + \frac{1}{\chif(G)}$ and $\chi(\mycielski(G)) = \chi(G) +1$. 
\end{lemma}
\begin{figure}
    \centering
    \begin{subfigure}[t]{0.4\textwidth}
        \centering
        \includegraphics[page=1]{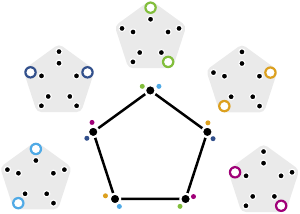}
        \caption{}
        \label{fig:examples_c5}
    \end{subfigure}
    \begin{subfigure}[t]{0.3\textwidth}
        \centering
        \includegraphics[page=2]{figures/examples_mycielski.pdf}
        \caption{}
        \label{fig:examples_mycielski_c5}
    \end{subfigure}
    \caption{(\subref{fig:examples_c5}) A $\by{5}{2}$-coloring of~$C_5$. (\subref{fig:examples_mycielski_c5}) The graph~$\mycielski(C_5)$ obtained from~$C_5$.}
    \label{fig:example_mycielski}
\end{figure}

\section{A lower bound on local bipartite covering numbers.}

In this chapter, we study local $\calB$-covers for the class~$\calB$ of all bipartite graphs.
In particular, we give a lower bound on the local $\calB$-covering number (cf. \cref{thm:fractional-chromatic-cover-bip}), from which we derive an alternative proof for a question raised by Fishburn an Hammer (cf. \cref{thm:local-compbip-kn}). 
Further, we prove $\NP$-completeness of the $2$-\probloc{$\calB$} (\cref{thm:local-bip-np}).

Recall that for the class~$\calB$ of all bipartite graphs the global $\calB$-covering number~$\cn{g}{\calB}{H}$ depends logarithmically on the chromatic number~$\chi(H)$.
In fact, for the class~$\calG_r$ of all graphs of chromatic number at most~$r$, we similarly obtain~$\cn{g}{\calG_r}{H} = \ceil{\log_r(\chi(H))}$ for every graph~$H$ \cite[Proposition~4.11]{stumpf2017ma}. 
For the local $\calG_r$ covering number, we provide a similar lower bound in the fractional chromatic number.
\begin{restatable}{theorem}{thmStructFractionalLower}
	\label{thm:fractional-chromatic-cover}
	If~$\calG$ is a graph class with $\chif(\calG) \leq r$ for some~$r \in \mathbb{R}_{>1}$, then we have $\log_r(\chif(H)) \leq \cl{\calG}(H)$ for every graph~$H$.
\end{restatable}
\begin{proof}[Proof of \cref{thm:fractional-chromatic-cover}]
	Consider a $k$-local $\calG$-cover $\varphi\colon G_1 \cupdot \dots G_t \to H$ of~$H$ with $k = \cn{l}{\calG}{H}$.
    As~$\chif(\calG) \leq r$, there is for each~$G_i$ an $\by{a_i}{b_i}$-coloring $\Phi_i \colon V(G_i) \to \binom{[a_i]}{b_i}$ with $\frac{a_i}{b_i} \leq \chi(G_i) \leq r$ (cf. \cref{thm:fract-chrom-rational} and \cref{obs:fract_chi_monotone}).
    We now construct a vertex-coloring~$\Phi$ of~$H$ where each vertex is assigned a subset of colors from~$A = [a_1] \times \cdot \times [a_t]$.
    A color~$(x_1, \dots, x_t) \in A$ lies in~$\Phi(v)$ for a vertex~$v \in V(H)$ if and only if for each~$i \in [t]$ the entry~$x_i$
    \begin{itemize}
    \item is a color of~$\Phi_i(v)$ when $v$ is hit by~$G_i$
    \item or any color of~$[a_i]$ otherwise,
    \end{itemize}
    see \cref{fig:example_fractional_chrom} for an example.
    Note that here the sizes of the sets~$\Phi(v)$ may differ.
    Yet, for each edge~$uv\in E(H)$ the color sets~$\Phi(u)$ and~$\Phi(v)$ are distinct.
    Indeed, as the edge~$uv$ is covered by some graph~$G_i$, we have $\Phi_i(u) \cap \Phi_i(v) = \varnothing$ while $u$ and~$v$ are both hit by~$G_i$.
    Thus, any two colors~$x \in \Phi(u)$ and~$y \in \Phi(v)$ differ in the $i$-th entry.
    That is, restricting all colors sets~$\Phi(v)$ to the same size yields an $\by{\abs{A}}{b}$-coloring~$\Phi'$ where $b$ is the smallest size among all sets~$\Phi(v)$.

	For a vertex $v \in V(H)$, consider the set $I_v = \{i \in [t] \mid v \in V(G_i)\}$.
	Note that by construction, we have $\abs{\Phi(v)} = \prod_{i \in [t] - I_v} a_i \prod_{i \in I_v} b_i$ and $\abs{A} = \prod_{i \in [t]} a_i$.
	Since the cover $\varphi$ is $k$-local and $\frac{a_i}{b_i} \leq r$ for every $i \in [t]$, we obtain
	\begin{align*}
		\frac{\abs{A}}{\abs{\Phi(v)}} = \prod_{i \in I_v} \frac{a_i}{b_i} \leq r^{\abs{I_v}} \leq r^k.
	\end{align*}
	As this holds in particular for the vertex with smallest color set, the $\by{\abs{A}}{b}$-coloring $\Phi'$ certifies $\chif(H) \leq r^k$ which implies ${\log_r(\chif(H)) \leq k = \cl{\calG}(H)}$.
\end{proof}

\begin{remark}
    The bound on the local covering number has first been shown by Schwebler in a preliminary work on which this paper is based \cite{schweblerGraphCoveringAlgorithnms2026}.
    Most recently, Gujgiczer, Marits, Ozeki independently obtained the same lower bound for the \emph{global} $\calG$-covering number~$\cg{\calG}(H)$ for the graph class~$\calG$ of graphs with fractional chromatic number at most~$r$ \cite{gujgiczerCoverNumbersGraph2026}. 
    As~$\cl{\calG}(H) \leq \cg{\calG}(H)$, the above theorem provides a strengthening.
\end{remark}
Note that \cref{thm:fractional-chromatic-cover-bip} directly follows from \cref{thm:fractional-chromatic-cover} as $\chif(G) \leq 2$ for every bipartite graph~$G$ (cf. \cref{obs:fract_chi_monotone}).
\fractionalChromaticCoverBip*
However, in contrast to the global setting, while the lower bound is tight (the complete graph~$K_n$ provides an example, see \cref{thm:local-compbip-kn}), in general it does not coincide with the local $\calB$-covering number (cf. \cref{thm:fractional-chromatic-cover-bip-not-tight}).

Fishburn and Hammer asked in \cite[p.\,147]{fishburn1996BipartiteDimension} whether $\cn{l}{\calB}{K_n} = \ceil{\log(n)}$ for every~$n$, a question which has previously been answered positively by Hansel \cite{hansel1964NombreMinimal} and Krichevskii \cite{Krichevskii1963ComplexityOC} and Radhakrishnan \cite{radhakrishnanen2001tropy}.
An explicit proof is presented by Dong and Liu \cite[Theorem~3.1]{dong2017decompositionIntoCompleteBipartite}, where it is proven explicitly for $\calB$-decompositions of~$K_n$ (the proof takes roughly two pages, but does not require any previous work or theory). The authors note that similar arguments yield the same result for covers.
In fact, \cref{thm:fractional-chromatic-cover-bip} also yields a proof for this result.
\DongLiuClCgCompbipKn*
\begin{proof}
	By \cref{thm:fract-chrom-clique} and \cref{obs:fract_chi_monotone}, we have $\chif(K_n) = n$, and thus \cref{thm:fractional-chromatic-cover-bip} yields $\ceil{\log_2(n)} \leq \cl{\calB}(K_n) \leq \cl{\calBc}(K_n)$ as each cover with complete bipartite graphs is also a bipartite cover.	
    To obtain the upper bound, we show that $\cg{\calBc}(K_n) = \cg{\calB}(K_n)$.
    The upper bound then follows from $\cl{\calBc}(K_n) \leq \cg{\calBc}(K_n) = \cg{\calB}(K_n) = \ceil{\log_2(n)}$ (see \cref{lem:cover_with_log_chi_bipartite_graphs}).

    Clearly, $\cg{\calB}(K_n) \leq \cg{\calBc}(K_n)$. 
    On the other hand, each $\calB$-cover~$\varphi\colon B_1 \cupdot \dots \cupdot B_t$ of~$K_n$ can be extended to a $\calBc$-cover by adding missing edges to each of the graphs~$B_i$.
    Indeed, the vertices of each graph~$B_i$ induce a complete bipartite graph~$B_i'$ of~$K_n$.
    That is, the graphs~$B_1', \dots, B_t'$ form a $\calBc$-cover of~$K_n$ and $\cg{\calB}(K_n) \geq \cg{\calBc}(K_n)$ follows.
\end{proof}
Note in particular, that above theorem shows that the lower bound of \cref{thm:fractional-chromatic-cover-bip} is tight as $\chif(K_n) = n$.

However, in general the lower bound obtained in \cref{thm:fractional-chromatic-cover} does not coincide with the local~$\calB$-covering.
To prove this, we show how the local $\calB$-covering number of a graph~$H$ changes when adding a \emph{universal vertex} to~$H$, that is by adding a vertex~$w$ connected to each vertex of~$H$.
We denote the resulting graph by~$H \join w$. 
\begin{lemma}
	\label{thm:clbip-join-k1}
    For every graph~$H$, we have~$\cl{\bip}(H \join w) \leq 2$ if and only if $\chi(H) \leq 3$.
\end{lemma}
\begin{proof}
	We start with the easy direction: If $\chi(H) \leq 3$, then $\chi(H \join w) \leq 4$ and thus we get $\cl{\bip}(H \join w) \leq \cg{\bip}(H \join w) = \ceil{\log(\chi(H \join w))} \leq 2$ by \cref{lem:cover_with_log_chi_bipartite_graphs}.
	
	For the other direction, suppose that $H' = H \join w$ has $2$-local $\bip$-cover $\varphi\colon B_1 \cupdot \dots \cupdot B_t \to H'$.
	We need to show that $H$ has a $3$-coloring.
    As~$\varphi$ is $2$-local, we may assume that all edges incident to the universal~$w$ are covered by~$B_1$ and~$B_2$.
	That is, each vertex of $H'$ is hit by at least one of~$B_1$ and~$B_2$.
	Thus, the graph $H' - E(B_1) - E(B_2)$ has a $1$-local $\bip$-cover and is therefore bipartite.
	It follows that the edges of $H' - E(B_1) - E(B_2)$ can be covered by a single graph~$B_3$.
	Without loss of generality, we may thus assume that $\varphi$ consists only of the graphs~$B_1, B_2, B_3$ and that each vertex of $H'$ is hit by exactly two guests, see \cref{fig:universal_cover}.
	
	For each $i \in [3]$, let $X_i$ and $Y_i$ denote the parts of $B_i$ such that $w \in X_1$ and $w \in X_2$.
	To every vertex~$v$ of $H'$, we assign a label~$p(v) = (p_1(v),p_2(v),p_3(v))$ based on the membership to the parts of the graphs~$B_i$ where
	\begin{align*}
		p_i(v) = \begin{cases}
			0 & \text{if } v \in X_i \\
			1 & \text{if } v \in Y_i \\
			\star & \text{if } v \notin V(B_i).
		\end{cases}
	\end{align*}
	By definition, we have $p(w) = (0,0,\star)$ and for each edge $uv \in E(H')$ there is an $i \in [3]$ with ${\{p_i(u), p_i(v)\} = \{0, 1\}}$.
    Recall that each vertex~$v$ of~$H$ is hit exactly twice and connected with an edge to~$w$.
    That is, exactly one of the entries~$p_i(v)$ is a $\star$, and at least one of the entries~$p_1(v)$ and~$p_2(v)$ equals~$1$.
    Thus, there are seven possible labels for vertices of~$H$: $(0,1,\star)$, $(1,0,\star)$, $(1,1,\star)$, $(1,\star,0)$, $(1,\star,1)$, $(\star,1,0)$ and $(\star,1,1)$.
	Consider the graph $L$ with these seven labels as vertices and an edge between the labels $(v_1, v_2, v_3)$ and $(u_1, u_2, u_3)$ if there is some $i \in [3]$ with $\{v_i, u_i\} = \{0, 1\}$.
	This graph $L$ has a $3$-coloring $\Psi\colon V(L) \to [3]$ as illustrated in \cref{fig:universal_coloring_auxiliary_graph}.
	From this we obtain a $3$-coloring $\Psi'\colon V(H) \to [3]$ of $H$ by setting $\Psi'(v) = \Psi(p(v))$, see \cref{fig:universal_coloring} for an example.
	\begin{figure}
        \begin{subfigure}[t]{0.3\textwidth}
            \centering
            \includegraphics[page=1]{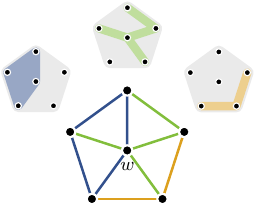}
            \caption{}
            \label{fig:universal_cover}
        \end{subfigure}\hfill
        \begin{subfigure}[t]{0.3\textwidth}
            \centering
            \includegraphics[page=2]{figures/example_universal.pdf}
            \caption{}
            \label{fig:universal_coloring}
        \end{subfigure}\hfill
        \begin{subfigure}[t]{0.3\textwidth}
            \centering
            \includegraphics[page=3]{figures/example_universal.pdf}
            \caption{}
            \label{fig:universal_coloring_auxiliary_graph}
        \end{subfigure}\hfill
		\caption{(\subref{fig:universal_cover}) A $2$-local cover of~$C_5 \join w$. (\subref{fig:universal_coloring}) The resulting proper $3$-vertex-coloring of~$C_5$ based on the labels of~$C_5 \join w$. (\subref{fig:universal_coloring_auxiliary_graph}) A proper $3$-vertex coloring of the graph~$L$ representing the labels.}
        \label{fig:universal_augmentation}
	\end{figure}
\end{proof}

\begin{corollary}
	\label{thm:fractional-chromatic-cover-bip-not-tight}
    There is a graph~$H$ such that $\cl{\bip}(H) > \ceil{\log(\chif(H))}$.
\end{corollary}
\begin{proof}
	Let $H = \mycielski(C_5) \join w$.
    As~$C_5$ admits a~$\by{5}{2}$-coloring (cf. \cref{fig:examples_c5}) and its largest independent set has size~$2$, we obtain $\chif(C_5) = \frac{5}{2}$ by \cref{thm:fract-chrom-ind}.
	By \cref{lem:mycielski_fractional_chrom}, we get $\chif(\mycielski(C_5)) = \frac{5}{2} + \frac{2}{5} = \frac{29}{10}$.
    Using ten new colors for the color set of~$w$ yields a $\by{39}{10}$-coloring of~$H$ certifying $\chif(H) \leq \frac{39}{10} \leq 4$.
	So we obtain a lower bound of $\ceil{\log(\chif(H))} \leq 2$.
	However, $\mycielski(C_5)$ is not $3$-colorable (cf. \cref{lem:mycielski_fractional_chrom}), thus \cref{thm:clbip-join-k1} yields $\cl{\bip}(H) > 2$.
\end{proof}
While above corollary shows that the lower bound obtained in \cref{thm:fractional-chromatic-cover-bip} is not tight, we do not know whether there exists an upper bound on the local $\calB$-covering number in terms of the fractional chromatic number, see \cref{discussion:fractional_chromatic} for a discussion.

In fact, the construction provided in \cref{thm:clbip-join-k1} directly yields $\NP$-hardness of the $2$-\probloc{$\calB$} by a reduction from the $3$-vertex coloring problem, which is known to be $\NP$-hard \cite{stockmeyer3-Colorability1973}.
$\NP$-membership follows from \cref{obs:NP-membership}.
\localBcoveringNPC*

\section{Complexity of covering problems with complete and complete bipartite graphs}
\begin{table}[ht]
	\ra{1.3}
	\def\smallDist{3}
	\def\largeDist{8}
	\small
	\centering
	\begin{subtable}[t]{\textwidth}
		\centering
		\begin{tabularx}{0.5\textwidth}{c *{2}{Y}}
        \toprule
		\arrayrulecolor{white}
			 & \multicolumn{1}{Y}{bicliques $\calBc$} & \multicolumn{1}{Y}{cliques $\compl$}  \\
			global & \multicolumn{1}{|c|}{\cellcolor{colorNP}\shortref{thm:global-clique-bc-np}, \cite{orlin1977contentment,mullerEdgePerfectnessClasses1996}} &
			\multicolumn{1}{|c|}{\cellcolor{colorNP}\shortref{thm:global-clique-bc-np}, \cite{orlin1977contentment,kouCoveringEdgesCliques1978}}  \\
			\cmidrule{1-3}
			local &  \multicolumn{1}{|c|}{\cellcolor{colorNP}\shortref{thm:local-bc-np}} &
			\multicolumn{1}{|c|}{\cellcolor{colorNPK}\shortref{thm:local-clique-np}, \cite{poljak1981complexityRepresentation}{\protect\footnotemark}}  \\
			\arrayrulecolor{black}\bottomrule
		\end{tabularx}
		
		\medskip
		\caption{\;\;  \textcolor{colorNP}{$\blacksquare$} weakly $\NP$-hard \;\; \textcolor{colorNPK}{$\blacksquare$} strongly $\NP$-hard }
	\end{subtable}
	\caption{Overview of the complexity results obtained in this chapter.
	The cells correspond to the \glob{$\calG$} and \loc{$\calG$} problem where $\calG$ depends on the column.
	Numbers $\langle X \rangle$ refer to the corresponding result.
	}
	\label{tab:partial}
\end{table}
\footnotetext{They show the stronger result that the $k$-\probloc{$\compl$} is $\NP$-hard for every fixed $k \geq 4$.}

The study of global and local $\calBc$- and $\calK$-covering numbers for the class~$\calBc$ of all complete bipartite and the class~$\calK$ of all complete graphs led Fishburn and Hammer to ask about the complexity of the corresponding decision problems \cite[p.\,148]{fishburn1996BipartiteDimension}.
While both global variants have subsequently been shown to be $\NP$-complete (cf. \cite[Theorem~8.1]{orlin1977contentment} \cite[Proposition~3]{kouCoveringEdgesCliques1978} for~$\calK$ and \cite[Theorem~6]{mullerEdgePerfectnessClasses1996}\cite[Theorem~8.1]{orlin1977contentment} for~$\calBc$), in the local setting, this has only been proven to be the case for the class~$\calK$ \cite[Theorem~2.1]{poljak1981complexityRepresentation}.
The main goal of this chapter is to prove that the same holds for the class $\biclique$ of all complete bipartite graphs.
\begin{restatable}{theorem}{localBicliqueNP}
	\label{thm:local-bc-np}
	The \probloc{$\biclique$} is \npc.
\end{restatable}
We present a framework-based approach within which we can reformulate Orlin's proof for the global $\calBc$- an $\calK$-covering problems using the same ideas as in their work \cite{orlin1977contentment}.
The framework evolves around \emph{$S$-partial $\calG$-covers}, see \cref{par:S-partial-covers} for a definition.
We prove $\NP$-hardness of the \textsf{global-} and \textsf{local-}$\calK$- and $\calBc$-\textsf{covering} problem via a reduction from the respective partial variant, see \cref{tab:partial} for an overview.
\thmGlobLocCompbipBip*

The approach unifies the proofs of the four results and the framework may be applicable to other graph classes~$\calG$.
Yet, note that the proofs given in \cite[Theorem~8.1]{orlin1977contentment} for showing $\NP$-completeness of the global $\calBc$- and~$\calK$-covering problem  (roughly one and a half pages), and in \cite[Theorem~2.1]{poljak1981complexityRepresentation} (roughly half a page) for the local $\calK$-covering problem are shorter.

Proving $\NP$-hardness of the \probloc{$\biclique$} turns out to be surprisingly difficult.
While our proof idea is essentially based upon Orlin's proof in the global setting, the local variant requires a more detailed analysis. 
To unify the proofs for the $\calK$- and the $\calBc$-covering problem, we use a framework approach for showing $\NP$-completeness of a $\calG$-covering problem which consists of two steps:
\begin{enumerate}[label=\enumstyle{(\roman*)}, ref=\roman*]
    \item proving $\NP$-completeness of the partial-$\calG$-covering problem (\cref{subsec:partial_covering_NP})
    \item constructing what we call a \emph{$\calG$-gadget} for the class~$\calG$ (\cref{subsec:bc-clique-global,subsec:bc-clique-local}).
\end{enumerate}
The properties of the gadget are such that attaching it to each optional edge~$e \in E(H)-S$ of a host graph~$H$ reduces an instance~$(H,S,k)$ of the partial-$\calG$-covering problem to an instance~$(H',k')$ of the $\calG$-covering problem.
The gadget ensures that in each $\calG$-cover of~$H'$ every optional edge~$e \in E(H)-S$ is covered by some guest~$G$ which covers no other edge of~$H$. 
We provide such a framework both for the global and the local setting.
To exemplify how the framework can be used and highlight the main proof ideas, we first consider the global setting (\cref{subsec:bc-clique-global}) before turning to the local setting (\cref{subsec:bc-clique-local}). 

\subsection{\texorpdfstring{$\bm{\NP}$}{NP}-hardness of partial covering problems}
\label{subsec:partial_covering_NP}
Orlin's $\NP$-hardness proofs of the global $\calBc$- and~$\calK$-covering problem both use their cor\-res\-pon\-ding partial variant as an intermediate step.
Indeed, it seems that showing $\NP$-hardness of the partial $\calG$-covering problem is easier.
In this section, we show that the \textsf{partial-global-} and \textsf{partial-local}-$\calK$ and~$\calBc$-\textsf{covering} problem are $\NP$-hard for the class~$\calK$ of all complete graphs and the class~$\calBc$ of all complete bipartite graphs both in the global and local setting. 

Recall that the \emph{clique-vertex-cover number} $\kappa(H)$ of a graph~$H$ is the smallest number of cliques in~$H$ that hit all vertices, i.e., $\kappa(H)$ is the minimum~$t$ such that there exists a (not necessarily edge-surjective) injective vertex-surjective graph homomorphism $\gamma\colon A_1 \cupdot \dots \cupdot A_t \to H$ with $A_i \in \calK$ for every~$i$ (called a \emph{clique-vertex-cover}).
As determining $\kappa(H)$ boils down to computing the chromatic number of its complement, deciding whether a graph~$H$ satisfies $\kappa(H) \leq k$ is $\NP$-complete.
We first show $\NP$-completeness of the global and local $\calK$-covering number.
The proof idea is similar to the $\NP$-hardness proof of the local $\calK$-covering problem provided by Poljak, Rödl and Turzík \cite[Theorem~2.1]{poljak1981complexityRepresentation}.

\begin{figure}[t]
	\centering
	\begin{subfigure}[t]{.45\linewidth}
		\centering
		\includegraphics[page=2]{figures/example_kappa.pdf}
		\caption{}
		\label{fig:proof-partial-clique-np-1}
	\end{subfigure}
	\begin{subfigure}[t]{.45\linewidth}
		\centering
		\includegraphics[page=1]{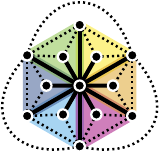}
		\caption{}
		\label{fig:proof-partial-clique-np-2}
	\end{subfigure}
	\caption{(\subref{fig:proof-partial-clique-np-1}) A clique-vertex-cover of a graph~$H$ certifying $\kappa(H) \leq 6$. (\subref{fig:proof-partial-clique-np-2}) An $S$-partial cover of~$H \join x$ where~$S$ consists of all edges incident to~$x$ (thick).}
	\label{fig:proof-partial-clique-np}
\end{figure}

\begin{lemma}
	\label{thm:partial-clique-cover-reduction}
	Let $H$ be a graph and $H' = H \join x$ be obtained by joining a universal vertex $x$ to $H$.
	Let $S = \{xv \mid v \in V(H)\}$ be the set of edges incident to the universal vertex $x$ and let $k \in \mathbb{N}$.
	The following statements are equivalent:
	\begin{enumerate}[label=\enumstyle{(\roman*)}, ref=\roman*]
		\item\label{itm:kappa} $\kappa(H) \leq k$,
		\item\label{itm:clique_local} $H'$ has a $k$-local $S$-partial $\mathcal{K}$-cover, and
		\item\label{itm:clique_global} $H'$ has a $k$-global $S$-partial $\mathcal{K}$-cover.
	\end{enumerate}
\end{lemma}
\begin{proof} 
	To show that \eqref{itm:kappa} implies \eqref{itm:clique_local} consider a clique-vertex-cover~$\gamma\colon A_1 \cupdot \dots \cupdot A_k \to H$ of~$H$.
	Adding $x$ to each of the cliques~$A_i$ yields $k$~cliques $A_1', \ldots, A_k'$ which form an $S$-partial $\compl$-cover~$\gamma'$ of~$H'$, see \cref{fig:proof-partial-clique-np}.
    As~$\gamma'$ is $k$-global, it is in particular $k$-local.

    Removing all guests which do not contain the vertex~$x$ from a $k$-local $S$-partial $\calK$-cover of~$H'$ yields a $k$-global $S$-partial $\calK$-cover~$\gamma'$ as all edges of~$S$ are incident to~$x$.
    That is, \eqref{itm:clique_local} implies \eqref{itm:clique_global}.
	
	It remains to show that \eqref{itm:clique_global} implies \eqref{itm:kappa}.
    Restricting a $k$-global $S$-partial $\mathcal{K}$-cover $\varphi'\colon A_1' \cupdot \dots \cupdot A_k' \to H'$ of $H'$ to the graph~$H$ yields a clique-vertex-cover~$\varphi$ of~$H$.
    Indeed, as each vertex of~$H$ is an endpoint of some edge of~$S$, $\varphi$ hits each vertex of~$H$ and $\kappa(H) \leq k$ follows.
\end{proof}
As deciding whether a graph~$H$ satisfies $\kappa(H) \leq k$ reduces to finding a $k$-global (respectively $k$-local) $S$-partial $\calK$-cover of~$H \vee x$ where~$S$ corresponds to the edges incident to~$x$, $\NP$-complete\-ness of the partial-global and partial-local $\calK$-covering problem follow (cf. \cref{obs:NP-membership}).
\begin{corollary}
	\label{thm:partial-global-clique-cover-np}
	The \probpartglob{$\compl$} and the \probpartloc{$\compl$} are $\NP$-complete.
\end{corollary}

As for complete graphs, the vertex-clique-cover problem can be reduced to the \probpartglob{$\calBc$}.
\begin{lemma}[{Orlin \cite[Theorem 8.1]{orlin1977contentment}}]
	\label{thm:partial-global-biclique-cover-np}
	The \probpartglob{$\biclique$} is \nph.
\end{lemma}

Now, we come to the \probpartloc{$\biclique$}.
This time, we do not reduce from the \textsf{clique-vertex-covering-problem}, but from the \probloc{$\bip$}.
Here, $\bip$ is the class of all bipartite graphs.

To show that the \probpartloc{$\biclique$} is $\NP$-complete, recall that the $2$-\probloc{$\bip$} is $\NP$-complete (cf. \cref{thm:local-bip-np}). 
Considering the graph~$H$ of an instance of the $2$-\probloc{$\bip$} as a subgraph of the complete graph~$K_n$ on the same number~$n$ of vertices yields an instance $(K_n,E(H),2)$ of the \probpartloc{$\biclique$}.
Clearly, every $E(H)$-partial $\calBc$-cover of~$K_n$ can be restricted to a $\calB$-cover of~$H$, and every $\calB$-cover of~$H$ can be extended (by adding missing optional edges) to an $E(H)$-partial $\calBc$-cover of~$K_n$, see \cref{fig:proof-partial-biclique-np} for an example.

\begin{figure}[t]
	\centering
	\begin{subfigure}[t]{.45\linewidth}
		\centering
		\includegraphics[page=1]{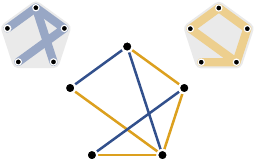}
		\caption{}
		\label{fig:proof-partial-biclique-np-1}
	\end{subfigure}
	\begin{subfigure}[t]{.45\linewidth}
		\centering
		\includegraphics[page=2]{figures/example_partial.pdf}
		\caption{}
		\label{fig:proof-partial-biclique-np-2}
	\end{subfigure}
	\caption{(\subref{fig:proof-partial-biclique-np-1}) A $2$-local $\calB$-cover of a graph~$H$ on five vertices. (\subref{fig:proof-partial-biclique-np-2}) The corresponding $E(H)$-partial $\calBc$-cover of the supergraph~$K_5$. Optional edges are represented by dotted lines.}
	\label{fig:proof-partial-biclique-np}
\end{figure}

\begin{lemma}
	\label{thm:local-partial-bc-np}
	The \probpartloc{$\biclique$} is $\NP$-complete.
\end{lemma}

\subsection{Global Covers}
\label{subsec:bc-clique-global}

\begin{figure}[t]
	\centering
	\begin{subfigure}[t]{.2\linewidth}
		\centering
		\includegraphics[page=3]{figures/examples_global_gadget.pdf}
		\caption{}
        \label{fig:global-gadget-bc-1}
	\end{subfigure}
	\begin{subfigure}[t]{.3\linewidth}
		\centering
		\includegraphics[page=2]{figures/examples_global_gadget.pdf}
		\caption{}
        \label{fig:global-gadget-bc-2}
	\end{subfigure}\hfill
	\begin{subfigure}[t]{.4\linewidth}
		\centering
		\includegraphics[page=1]{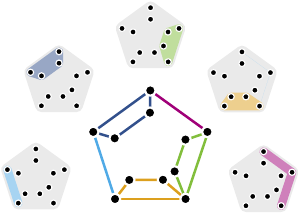}
		\caption{}
        \label{fig:global-gadget-bc-3}
	\end{subfigure}
	\caption{(\subref{fig:global-gadget-bc-1}) A global $\calBc$-gadget~$(C_4,e)$ for the class~$\calBc$ of all complete bipartite graphs. (\subref{fig:global-gadget-bc-2}) A $2$-global $S$-partial $\calBc$-cover of~$C_5$ with three optional edges (dotted). (\subref{fig:global-gadget-bc-3}) A $5$-global $\calBc$-cover of the graph~$H$ obtained from~$C_5$ by attaching the gadget~$(C_4,e)$ to each of the three optional edges.}
	\label{fig:global-gadget-bc}
\end{figure}

In this section we describe a framework for showing $\NP$-completeness of global $\calG$-covering problems.
The idea consists of reducing the partial variant to the general variant by \emph{attaching} a gadget~$(F,e)$ to each optional edge of a given partial instance.
A gadget~$(F,e)$ consists of a graph~$F$ and a (directed) edge~$e \in E(F)$.
Attaching~$(F,e)$ to a (directed) edge~$e' \in E(H')$ consists of taking the disjoint union of~$F$ and~$H'$ and identifying the edges~$e$ and~$e'$, see \cref{fig:global-gadget-bc} for an illustration.
The resulting graph is an \emph{extension} of~$(F,e)$.
As all gadgets~$(F,uv)$ we consider are symmetric with respect to~$uv$ (that is attaching~$(F,uv)$ or $(F,vu)$ to a graph~$H'$ yields the same extension), we omit the orientation of~$e$.
\begin{definition}
	\label{def:global-partial-reduction-gadget}
For a graph class~$\calG$, we call a gadget~$(F,e)$ a \emph{global $\calG$-gadget} if
\begin{myenum}{Glob-}
		\item\label{itm:part-glob-coverable} $F \in \calG$, and
		\item\label{itm:part-glob-useful} each $\calG$-cover $\varphi\colon G_1 \cupdot \dots \cupdot G_t \to H$ of every extension~$H$ of~$F$ contains a guest $G_i$ which is a subgraph of $F$.
\end{myenum}
\end{definition}

Using such gadgets, we reduce a global-partial $\calG$-covering problem to a global $\calG$-covering problem.

\begin{theorem}
	\label{thm:global-partial-reduction}
	Let $\calG$ be a hereditary graph class such that the \probpartglob{$\calG$} is \nph.
	If there is a global-partial-reduction-gadget $(F, e)$ for the class~$\calG$, then the \probglob{$\calG$} is \nph.
\end{theorem}

Before proving above theorem, we illustrate how it can be applied to show~$\NP$-completeness for the global $\calBc$- and $\calK$-covering problem.
Both the class~$\calBc$ and the class~$\calK$ admit a global gadget.
In the case of complete bipartite graphs, this turns out to be $(C_4,ab)$ where $C_4$ is a cycle on four vertices~$a,b,c,d$ (in order), see \cref{fig:global-gadget-bc}.\footnote{Orlin uses the same gadget to reduce the global-partial $\calBc$-covering problem to the global problem \cite[Theorem 8.1]{orlin1977contentment}.}
Clearly, $C_4 \in \calBc$.
Further, no vertex of~$H-C_4$ is adjacent to~$c$ or~$d$. 
Thus, no complete bipartite graph~$G_i$ of a $\calBc$-cover~$\varphi\colon G_1 \cupdot \dots \cupdot G_t \to H$ of an extension~$H$ of~$(C_4,ab)$ covers both an edge in~$H-C_4$ and the edge~$cd$.
That is, the graph~$G_i$ that covers~$cd$ lies completely within~$C_4$ and \eqref{itm:part-glob-useful} is satisfied.
Similarly, we see that the triangle~$K_3$ is a global $\calK$-gadget with respect to the class~$\calK$ of all complete graphs.
Therefore, as both global-partial variants are $\NP$-complete (cf. \cref{thm:partial-global-clique-cover-np,thm:partial-global-biclique-cover-np}) and the global variants lie in~$\NP$ (cf. \cref{obs:NP-membership}), $\NP$-completeness follows for the global variants.

\begin{theorem}[{Orlin \cite[Theorem 8.1]{orlin1977contentment}}]
	\label{thm:global-clique-bc-np}
	The \textsf{global}-$\calK$- and $\calBc$-\textsf{covering} problem are $\NP$-complete for the class~$\calK$ of all complete graphs and the class~$\calBc$ of all complete bipartite graphs.
\end{theorem}

It remains to prove~\cref{thm:global-partial-reduction}.
We show $\NP$-completeness of the global variant via a reduction from the partial-global variant.
\begin{proof}[Proof of \cref{thm:global-partial-reduction}]
Given an instance~$(H,S,k)$ of the \probpartglob{$\calG$}, consider the graph~$H'$ obtained from~$H$ by attaching a global $\calG$-gadget to each optional edge~$e \in E(H)-S$ of~$H$.
It now suffices to prove that~$H'$ admits a $k'$-global $\calG$-cover with $k' = k+\abs{E(H)}-\abs{S}$ if and only if~$H$ has a $k$-global $S$-partial $\calG$-cover, see \cref{fig:global-gadget-bc} for an illustration.
In every $t$-global $\calG$-cover~$\varphi\colon G_1 \cupdot \dots \cupdot G_{t} \to H'$ of~$H'$, each gadget~$G$ attached to an optional edge is partially covered with a distinct graph~$G_i$ that covers no edge of~$H'-E(G)$ by \eqref{itm:part-glob-useful}.
Thus, removing these cover graphs~$G_i$ and restricting the resulting partial cover to~$H$ yields a $(t-\abs{E(H)]}+\abs{S})$-global $S$-partial $\calG$-cover since~$\calG$ is hereditary.
Further, every $k$-global $S$-partial $\calG$-cover can be extended to a $k'$-global $\calG$-cover of~$H'$ by covering each attached gadget separately due to~\eqref{itm:part-glob-coverable}.
\end{proof}

\subsection{Local Covers}
\label{subsec:bc-clique-local}
Within this section, we adapt the framework for global $\calG$-covering problems developed in \cref{subsec:bc-clique-global} to the local setting.
In the global setting, attaching a gadget to on optional edge~$e$ is similar to turning this edge into a non-optional edge while increasing the number of guests in an~$(S \cup \set{e})$-partial $\calG$-cover of minimum size by one.
In the local setting we aim to increase the hitcount of the endpoints of~$e$ by~$1$ instead.
\begin{definition}
	\label{def:local-partial-reduction-gadget}
	For a graph class~$\calG$ and~$k \in \N$, we call a gadget~$(F,e)$ a \emph{$k$-local $\calG$-gadget} if 
	\begin{myenum}{Loc-}
		\item \label{itm:local-partial-reduction-gadget-coverable} There is a $k$-local $\calG$-cover of $F$ which hits the endpoints of $e$ only once, and
		\item \label{itm:local-partial-reduction-gadget-subgraph} each $k$-local $\calG$-cover $\varphi\colon G_1 \cupdot \dots \cupdot G_t \to H$ of every extension of~$F$ contains a guest $G_i \subseteq F$ covering~$e$.
	\end{myenum}
\end{definition}
Yet, in contrast to the global setting, the gadgets we use to reduce an instance $(H,S,k)$ of the \probpartloc{$\calG$} to the \probloc{$\calG$} depend upon~$k$.
To illustrate the definition, we first construct gadgets for the class~$\calK$ of all complete graphs, see \cref{fig:partial-local-K-gadget} for an illustration.
\begin{figure}
    \centering
    \begin{subfigure}[t]{0.25\textwidth}
        \centering
        \includegraphics[page=1]{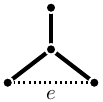}
        \caption{}
        \label{fig:partial-local-K-gadget-1}
    \end{subfigure}
    \begin{subfigure}[t]{0.3\textwidth}
        \centering
        \includegraphics[page=2]{figures/example_partial_gadget.pdf}
        \caption{}
        \label{fig:partial-local-K-gadget-2}
    \end{subfigure}
    \begin{subfigure}[t]{0.3\textwidth}
        \centering
        \includegraphics[page=3]{figures/example_partial_gadget.pdf}
        \caption{}
        \label{fig:partial-local-K-gadget-3}
    \end{subfigure}
    \caption{(\subref{fig:partial-local-K-gadget-1}) A $2$-local $\calK$-gadget~$(S_2,e)$ for the class~$\calK$ of all complete graphs. (\subref{fig:partial-local-K-gadget-2}) A $2$-local $S$-partial $\calK$-covering of a graph~$H$. Optional edges are represented by dotted lines. (\subref{fig:partial-local-K-gadget-3}) Attaching to each optional edge of~$H$ the gadget represented in (\subref{fig:partial-local-K-gadget-1}) yields a graph~$\hat{H}$ with a $3$-local $\calK$-covering. Note that the hitcount of each vertex incident to an optional edge of~$H$ increased by~$1$.}
    \label{fig:partial-local-K-gadget}
\end{figure}
\begin{lemma}
	\label{thm:gadget-clique-local-partial-reduction}
    For~$k \in \N$ the pair~$(S_k,e)$ where~$S_k$ is the graph obtained by joining two leaves of the star~$K_{1,k+1}$ with an edge~$e$ is a $k$-local $\calK$-gadget for the class~$\calK$ of all complete graphs.
\end{lemma}
\begin{proof}
    Clearly, covering the triangle induced by the center~$c$ of~$K_{1,k+1}$ and the endpoints of~$e$ and covering all other edges separately with a~$K_2$ yields a $\calK$-cover of~$S_k$ certifying \eqref{itm:local-partial-reduction-gadget-coverable}.

    To show that~$(S_k,e)$ also satisfies \eqref{itm:local-partial-reduction-gadget-subgraph} consider a $k$-local $\calK$-cover~$\varphi$ of an extension~$H$ of~$S_k$. 
    Note that the triangle~$T$ formed by the center~$c$ and the endpoints of~$e$ is the only clique of size at least~$3$ in~$H$ containing $c$. 
    While~$c$ has degree~$k+1$, it is only hit at most~$k$ times. 
    Thus, $T$ is part of the cover~$\varphi$ and \eqref{itm:local-partial-reduction-gadget-subgraph} follows.
\end{proof}
For the class $\biclique$ of complete bipartite graphs the construction of a $k$-local $\calBc$-gadget turns out to be more complicated.
In fact, the construction presented in \cref{subsec:local_bc_covers} is algebraic and only works for certain $k$.
Thus, we shall not require the existence of a $k$-local $\calG$-gadget for every $k$ in order to reduce the \probpartloc{$\calG$} to the \probloc{$\calG$}.
We only need to ensure that a $k$-local $\calG$-gadget exists for sufficiently many $k$. 
\begin{definition}
	Let $\calG$ be a graph class.
	For a set $K \subseteq \mathbb{N}$, we call $\set{(F_k, e_k) \given k \in K}$ \emph{$K$-local $\calG$-gadgets} if the following properties are satisfied for some~$\beta \in \mathbb{N}$:
	\begin{myenum}{LF-}
		\item\label{itm:partial-local-family-gadget} $(F_k, e_k)$ is a $k$-local $\calG$-gadget for every $k \in K$.
		\item\label{itm:partial-local-family-poly-size} Given $k \in K$, it is possible to construct $(F_k, e_k)$ in $\mathcal{O}(k^\beta)$ time.
		\item\label{itm:partial-local-family-poly-k} Given $\ell \in \mathbb{N}$, it is possible to find a $k \geq \ell$ in $\mathcal{O}(\ell^\beta)$ time which satisfies $k \in K$ and $k \in \mathcal{O}(\ell^\beta)$.
	\end{myenum}
\end{definition}

\begin{theorem}
	\label{thm:local-partial-reduction}
	Let $\calG$ be a hereditary guest class such that the \probpartloc{$\calG$} is \nph.
	If there are $K$-local $\calG$-gadgets~$\set{(F_k, e_k) \given k \in K}$ for some $K \subseteq N$, then the \probloc{$\calG$} is \nph.
\end{theorem}

Before proving above theorem, note that \cref{thm:gadget-clique-local-partial-reduction} and \cref{obs:NP-membership} now yield an alternative proof of $\NP$-completeness for the \probloc{$\calK$}.
\begin{corollary}
	\label{thm:local-clique-np}
	The \probloc{$\compl$} is $\NP$-complete.
\end{corollary}
However, the result presented in \cite[Theorem~2.1]{poljak1981complexityRepresentation}\footnote{The proof given by Poljak, Rödl and Turzík is very short (it requires half a page). It reduces the determination of the chromatic number to the \probloc{$\calK$}: Given a graph~$G$ we construct a graph~$H$ that it obtained from the complement of~$G$ by adding~$\abs{V(G)}$ new independent vertices~$y_1, \dots, y_n$, and adding a universal vertex~$x$ connected to all of~$V(G)$ and each vertex~$y_i$. It then suffices to observe that $\cl{\calK}(H) = \chi(G) + n$.} provides a generalization as Poljak, Rödl and Turzík show that the $k$-\probloc{$\calK$} is $\NP$-complete for \emph{every fixed}~$k \geq 4$ while it admits a polynomial time solution for $k = 2$.

The proof idea of~\cref{thm:local-partial-reduction} is very similar to the global version (cf. \cref{thm:global-partial-reduction}). 
Our aim is to transform an instance~$(H,S,k)$ of the \probpartloc{$\calG$} into an instance~$(L,\ell)$ of the \probloc{$\calG$} by attaching gadgets to optional edges.
Local gadgets increase the hitcount of the edge they are attached to by one.
Thus, as we want to increase the hitcount of every edge by exactly the same amount, we first add optional edges to ensure that each vertex of~$H$ is incident to the same number of optional edges.
The instance~$(L,\ell)$ then arises by attaching a gadget to each optional edge.

\begin{proof}[Proof of~\cref{thm:local-partial-reduction}]
    Given an instance~$(H,S,k)$ of the \probpartloc{$\calG$}, we construct an instance~$(L,\ell)$ of the \probloc{$\calG$} such that $H$ admits a $k$-local $S$-partial $\calG$-cover if and only if~$L$ has an $\ell$-local $\calG$-cover.
    $\NP$-hardness of the \probloc{$\calG$} then follows as the partial variant is $\NP$-hard by assumption.

    \proofsubparagraph*{Construction.} By \eqref{itm:partial-local-family-poly-k}, we can compute in polynomial time an integer~$\ell \geq k + \abs{V(H)}$ with $\ell \in K$ and whose size is polynomial in $k+\abs{V(H)}$.
    Note that each vertex of~$H$ is incident to at most~$\abs{V(H)} \leq \ell - k$ optional edges.
    We now turn~$H$ into an instance~$(H',S,\ell)$ of the \probpartloc{$\calG$} with $H \subseteq H'$ where each vertex of~$H$ is incident to exactly $\ell- k$ of optional edges.
    This can be achieved by adding up to~$\ell-k$ leaf vertices for each vertex~$v \in V(H)$ and joining them with an optional edge to~$v$.  
    We now obtain the graph~$L$ by attaching a copy of the gadget~$(F_{\ell},e_{\ell})$ to each optional edge of~$H'$, that is to all optional~$E(H)-S$ and the newly added edges~$E(H')-E(H)$.
    Note that the construction can be performed in polynomial time by \eqref{itm:partial-local-family-poly-size}.

    \proofsubparagraph*{Correctness.} It remains to prove that~$H$ admits a $k$-local $S$-partial $\calG$-cover if and only if~$L$ has an $\ell$-local $\calG$-cover.
    First assume that $H$ admits a $k$-local $S$-partial $\calG$-cover~$\gamma$. 
    By \eqref{itm:local-partial-reduction-gadget-coverable}, the gadget~$(F_{\ell},e_{\ell})$ has an $\ell$-local $\calG$-cover which hits the endpoints of~$e_{\ell}$ only once.
    Recall that each vertex~$v \in V(H)$ is incident to exactly $\ell-k$~optional edges.
    Thus, covering each copy of the gadget separately, we can extend~$\gamma$ to an $\ell$-local $\calG$-cover of~$L$.

    Now suppose~$L$ has an~$\ell$-local $\calG$-cover~$\varphi\colon G_1 \cupdot \dots \cupdot G_t \to L$.
    By \eqref{itm:local-partial-reduction-gadget-subgraph} there is for every optional edge~$e \in E(H')$ a guest~$G_i$ of~$\varphi$ covering~$e$ which lies within the copy~$F_e$ of the gadget attached to~$e$. 
    Removing all these graphs from~$\varphi$ and restricting the remaining guests to~$H$ yields an $S$-partial $\calG$-cover~$\gamma$ of of~$H$ as~$\calG$ is hereditary.
    As each vertex of~$H$ is incident to exactly $\ell-k$~optional edges in~$H'$, we reduce the hitcount of vertices in~$H$ by~$\ell-k$. 
    Thus, $\gamma$ is $k$-local.
\end{proof}

\subsubsection{Local Covers with Complete Bipartite Graphs}
\label{subsec:local_bc_covers}

The goal of this section is to construct $K$-local $\calBc$-gadgets for the class $\biclique$ of all complete bipartite graphs.
With \cref{thm:local-partial-reduction} and \cref{thm:local-partial-bc-np} $\NP$-completeness of the \probloc{$\calBc$} then follows.
\localBicliqueNP*

The construction of such a family~$\calF$ of gadgets relies in particular on increasing the hitcount of vertices by one.
This can be achieved with \emph{star-blowups}.
The \emph{$s$-star-blowup}~$S$ of a graph~$H$ is the graph consisting of the disjoint union of $s$~copies~$H_1, \dots, H_s$ of~$H$ and a universal vertex~$x$ joined with edges to all vertices of the graphs~$H_i$, see \cref{fig:star-blowup} for an example. 
We call~$x$ the \emph{center} of~$S$ and the graphs~$H_i$ its \emph{leaf-blocks}.
\begin{figure}
    \centering
        \centering
        \includegraphics[page=1]{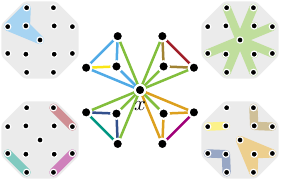}
        \caption{A $3$-local $\calBc$-cover of the $4$-star-blowup of~$K_3$ with center~$x$. Note that there is a leaf-block that is edge-disjoint from all guests hitting~$x$.}
        \label{fig:star-blowup}
\end{figure}

\begin{lemma}
	\label{lem:star-blowup-pigeon}
    Let $S$ be an $s$-star-blowup with center~$x$ of a graph~$H$ and assume~$S$ is an induced subgraph of a graph~$F$.
    In every $k$-local $\biclique$-cover~$\varphi\colon B_1 \cupdot \dots \cupdot B_t \to F$ of $F$ there are at least $s-k$ leaf-blocks of $S$ which are edge-disjoint from the guests~$G_i$ hitting $x$.
\end{lemma}
\begin{proof}
    As there are no edges joining two distinct leaf-blocks of~$S$, each complete bipartite graph~$B_i$ of a $k$-local cover~$\varphi\colon B_1 \cupdot \dots B_t \to F$ of~$F$ only covers edges of at most one leaf-block.
    Yet, there are $s$ leaf-blocks in total and the vertex~$x$ is hit at most $k$~times.
	Thus, there are at least $s-k$ leaf-blocks whose edges are not covered by any guest hitting~$x$.
\end{proof}

\begin{lemma}
	\label{thm:bc-gad1}
    For the $(k+1)$-star-blowup~$S$ of a graph~$H$ with $\cl{\biclique}(H) = k$, we have $\cl{\biclique}(S) = k+1$.
\end{lemma}

\begin{proof}
    Covering all edges incident to the center~$x$ of~$S$ with a single star, we can extend a $k$-local $\calBc$-cover of the leaf-blocks (each of which is a copy of~$H$) to a $(k+1)$-local $\calBc$-cover of~$S$, in particular we have $\cl{\biclique}(S) \leq k+1$.
	
	To obtain the lower bound, suppose there is a $k$-local $\calBc$-cover~$\varphi\colon B_1 \cupdot \dots \cupdot B_t$ of~$S$. 
    By \cref{lem:star-blowup-pigeon}, there is a leaf-block $\hat{H}$ of $H'$ which is edge-disjoint from all guests~$B_i$ hitting $x$.
    Yet, the edge~$vx$ of each vertex~$v \in V(\hat{H})$ is covered by a guest~$B_i$ which hits both~$x$ and~$v$.
    Thus, restricting~$\varphi$ to the leaf-block~$\hat{H}$ yields a $(k-1)$-local $\calBc$-cover of~$\hat{H}$, a contradiction to $\cl{\biclique}(H) = k$.
\end{proof}
\begin{corollary}
\label{cor:k-local-Bc-hard-then-k+1}
    If the $k$-\probloc{$\calBc$} is $\NP$-hard for some~$k \in \N$, then it is $\NP$-hard for all $k' \geq k$.
\end{corollary}
However, we only show that the \probloc{$\calBc$} is $\NP$-hard, that is when~$k$ is part of the input (cf. \cref{thm:local-bc-np}).

\cref{thm:bc-gad1} illustrates how star-blowups may increase the hitcount, but does not provide the final construction of the gadgets for~$\calBc$ which involves \emph{$\calBc$-blowups}. Let~$X'$ and~$Y'$ be the disjoint union of~$k+2$ and~$2k+3$ copies of a graph~$H$ respectively. The $k$-$\calBc$-blowup~$F$ is obtained from $X'$ and~$Y'$ by adding two vertices~$x,y$ and joining all vertices of~$X',x$ with edges to all vertices of~$Y',y$, see \cref{fig:bc-gad2} for an illustration.
\begin{figure}
		\centering
        \begin{subfigure}[t]{0.5\textwidth}
		\includegraphics[page=2]{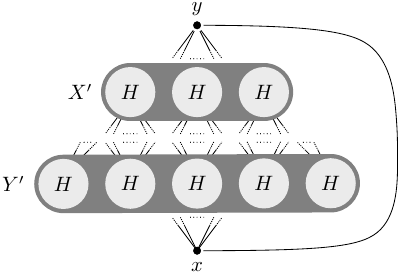}
        \caption{}
        \label{fig:bc-gad2}
        \end{subfigure}\quad\quad
        \begin{subfigure}[t]{0.4\textwidth}
		\includegraphics[page=4]{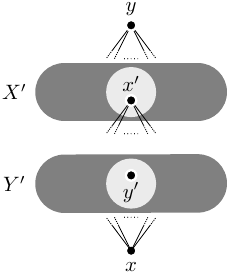}
        \caption{}
        \label{fig:bc-gad-xy}
        \end{subfigure}
		\caption{(\subref{fig:bc-gad2}) A $k$-$\calBc$-blowup~$F$ of a graph~$H$ with support~$xy$ for~$k=1$. Contracting each of the copies of~$H$ to a single vertex yields a complete bipartite graph where one part consists of~$y$ and~$2k+3$ other vertices, the other part of~$x$ and~$k+2$ other vertices. (\subref{fig:bc-gad-xy}) In every~$(k+1)$-local $\calBc$-cover of~$F-xy$, there is a vertex~$x'$ in~$X'$ that is hit~$k$ times by guests that do not cover the edge~$yx'$, and a vertex~$y'$ in~$Y'$ that is hit~$k$ times by guests that neither cover~$x'y'$ nor~$y'x$.}
\end{figure}
We call the edge~$xy$ the \emph{support} of~$F$.
Note that here the graph induced by~$y$ and~$X'$, as well as the graph induced by any vertex~$x' \in V(X')$ and a copy~$Y_i'$ of~$H$ within~$Y'$ forms a star-blowup of~$H$. 
\begin{lemma}
\label{lem:bc-local-gadget}
    Let~$F$ be the~$k$-$\calBc$-blowup of an $n$-vertex graph~$H$ with support~$xy$.
    If~$\cl{\calBc}(H) = k$, then $(F,xy)$ is a $(k+1)$-local $\calBc$-gadget for the class~$\calBc$ of all complete bipartite graphs.
    Given~$H$, the graph~$F$ can be constructed in time~$\mathcal{O}(k^2n^2)$ and has order~$\mathcal{O}(kn)$.
\end{lemma}
\begin{proof}
    Let~$X_1', \dots, X_{k+2}'$ and~$Y_1',\dots, Y_{2k+3}'$ denote the copies of~$H$ in~$X'$ and~$Y'$ respectively.
    Clearly, $F$ is a graph on~$\mathcal{O}(kn)$~vertices and can be constructed in time~$\mathcal{O}(k^2n^2)$ from~$H$.
    
    To show that~$(F,xy)$ satisfies \eqref{itm:local-partial-reduction-gadget-coverable}, we need to construct a $(k+1)$-local $\biclique$-cover hitting the vertices $x$ and $y$ only once.
    Covering each copy~$X_i'$ and~$Y_j'$ separately with a $k$-local cover of~$H$ yields a $k$-local $\calBc$-cover~$\varphi'\colon B_1 \cupdot \dots \cupdot B_t \to X' \cup Y'$ of~$X' \cup Y'$.
    The remaining edges form a complete bipartite graph~$B$, see \cref{fig:bc-gad2}.
    That is, extending $\varphi'$ by adding~$B$ as guest yields a $(k+1)$-local $\calBc$-cover $\varphi\colon B \cupdot B_1 \cupdot \dots \cupdot B_t \to F$ of~$F$ certifying \eqref{itm:local-partial-reduction-gadget-coverable}.
	
	To prove \eqref{itm:local-partial-reduction-gadget-subgraph}, we show that $\cl{\calBc}(F-xy) \geq k+2$.
    \eqref{itm:local-partial-reduction-gadget-subgraph} then follows.
    Indeed, consider a $(k+1)$-local-$\biclique$-cover~$\varphi\colon B_1 \cupdot \dots \cupdot B_t \to Q$ of an extension~$Q$ of~$F$.
    As~$F$ is an induced subgraph of~$Q$, the restriction~$\varphi_{\mid F}$ of~$\varphi$ to $F$ is a $\biclique$-cover.
    If each of the complete bipartite graphs~$\varphi(B_i) \cap F$ of~$\varphi_{\mid F}$ induces a complete bipartite graph in~$F-xy$, the restriction of~$\varphi$ to~$F-xy$ would yield a $(k+1)$-local $\calBc$-cover, a contradiction to $\cl{\calBc}(F-xy) \geq k+2$.
    Thus, there is some guest~$B_i$ such that~$\varphi(B_i) \cap F$ is a complete bipartite graph, but $\varphi(B_i) \cap (F-xy)$ is not.
    That is,~$B_i$ covers a cycle on vertices~$y,x',y',x$ (in order) with $x' \in V(X')$ and $y' \in V(Y')$. 
    As neither~$x'$ nor~$y'$ is adjacent to any vertex of~$Q-F$, the guest~$B_i$ lies completely within~$F$ and \eqref{itm:local-partial-reduction-gadget-subgraph} is satisfied.

    It remains to prove that $\cl{\calBc}(F-xy) \geq k+2$.
    Suppose there is a $(k+1)$-local $\calBc$-cover~$\varphi\colon B_1 \cupdot \dots \cupdot B_t \to F-xy$ for some~$\ell \in \N$.
    We show that
    \begin{enumerate}[label=\enumstyle{(\roman*)}, ref=\roman*]
        \item\label{prop:x'} there exists an index~$r \in [k+2]$ and a vertex~$x' \in V(X_r')$ such that $x'$ is hit at least~$k$~times by guests~$B_i$, each of which covers an edge within~$X_r'$ and does not cover the edge~$yx'$,
        \item\label{prop:y'} there exists an index~$s \in [2k+3]$ and a vertex~$y' \in V(Y_s')$ such that~$y'$ is hit at least~$k$ times by guests~$B_i$ that cover neither the edge~$x'y'$ nor~$y'x$,
        \item\label{prop:no_guest_hitting_x_x'} no guest~$B_i$ hits both~$x'$ and~$x$, 
    \end{enumerate}
    see \cref{fig:bc-gad-xy} for an illustration.
    As there is no guest that covers both the edge~$x'y'$ and~$y'x$ by \eqref{prop:no_guest_hitting_x_x'}, it now follows with \eqref{prop:y'} that $y'$ is hit at least~$k+2$~times, a contradiction to~$\varphi$ being $(k+1)$-local.

    To show \eqref{prop:x'}, observe that the graph~$F'$ induced by~$y$ and~$X'$ is a $(k+2)$-star-blowup of~$H$.
    The restriction of~$\varphi$ to~$F'$ is also a $(k+1)$-local $\calBc$-cover.
    Thus, by \cref{lem:star-blowup-pigeon}, there is a copy~$X_r'$ such that no guest~$B_i$ covers both edges within~$X_r'$ and edges connecting~$y$ to~$X_r'$.
    Since~$\cl{\calBc}(X_r') = \cl{\calBc}(H) = k$, there exists a vertex~$x' \in X_r'$ that is hit at least~$k$ times by guests~$B_i$ that cover edges within~$X_r'$, and none of which covers the edge~$yx'$.

    We now prove \eqref{prop:y'}. 
    Note that the graph induced by~$Y'$ and~$x$ is a $(2k+3)$-star-blowup of~$H$.
    The same argument as above shows that there is a set~$S \subseteq [2k+3]$ with~$\abs{S} = k+2$ such that no guest~$B_i$ covers both edges within~$Y_j'$ and edges connecting~$Y_j'$ to~$x'$ for any~$j \in S$.
    Similarly, as the graph induced by~$x'$ and the copies~$Y_j'$ with~$j \in S$ is a $(k+2)$-star-blowup of~$H$, there is an index~$s \in S$ such that no guest~$B_i$ covers both edges within~$Y_s'$ and the edge~$x'y'$. 
    Thus, as~$\cl{\calBc}(Y_s') = \cl{\calBc}(H) = k$, there is a vertex~$y' \in V(Y_s')$ that is hit at least $k$~times by guests~$B_i$ none of which covers the edge~$x'y'$ or~$y'x$.

    To prove \eqref{prop:no_guest_hitting_x_x'}, suppose some guest~$B_j$ covers both~$x'y'$ and~$y'x$. 
    As~$B_j$ is a complete bipartite graph, it cannot also cover the edge~$yx'$ since~$\varphi(B_j) \subseteq F-xy$, nor any edge within~$X_r'$.
    That is, by \eqref{prop:x'}, the vertex~$x'$ is hit at least~$k+2$ times: once by the guest~$B_j$, once by a guest covering the edge~$yx'$ and~$k$ times by other guests~$B_i$, a contradiction.
\end{proof}

\cref{thm:bc-gad1} provides a construction of graphs~$H$ with $\cl{\biclique}(H) = k$ for every~$k \in \N$, which results with \cref{lem:bc-local-gadget} in $(k+1)$-local $\calBc$-gadgets. 
Yet, their size grows exponentially with~$k$.
To obtain $K$-local $\calBc$-gadgets we need such graphs~$H$ of polynomial size \eqref{itm:partial-local-family-poly-size}. 
Erd\H{o}s, R\'enyi and S\'os \cite[Theorem~1]{erdosProblemGraphTheory1966}\footnote{The construction provided by Erd\H{o}s, R\'enyi and S\'os yields for every prime power~$p^d$ a $C_4$-free dense graph where every vertex has degree~$p^d$ or~$p^d+1$. When~$p$ is prime, there also exists a vertex of degree~$p+1$. While both properties on the degree are not part of the theorem statement, Erd\H{o}s, R\'enyi and S\'os state them explicitly within the proof.} construct for every prime~$p$ a dense graph~$G^{(p)}$ which contains no $4$-cycle~$C_4$ as a subgraph. 
Every $\calBc$-cover of such a graph~$G^{(p)}$ is a cover which only consists of stars.
Yet, stars only contain a linear number of edges.
Thus, if~$G^{(p)}$ is dense, the local $\calBc$-covering number of~$G^{(p)}$ is large.
\begin{lemma}[{Erd\H{o}s, R\'enyi, S\'os \cite[Theorem~1]{erdosProblemGraphTheory1966}}]
	\label{thm:c4-free}
	For every prime $p$, there exists a graph~$G^{(p)}$ on $p^2+p+1$~vertices such that 
	\begin{enumerate}[label=\enumstyle{(\roman*)}, ref=\roman*]
		\item\label{itm:c4-free-c4} $G^{(p)}$ does not contain a $C_4$, and
		\item\label{itm:c4-free-deg} every vertex $v \in V(G^{(p)})$ has degree $\deg(v) \in \{p-1, p\}$ and there is at least one vertex of degree $p$.
	\end{enumerate}
\end{lemma}

Knauer and Ueckerdt show that the \emph{star arboricity}~$\cl{\calS}(H)$ (i.e., the local $\calS$-covering number for the class~$\calS$ of all stars) is bounded from above and below in terms of what is known as the \emph{pseudoarboricity}~$p(H) = \max_{X \subseteq V(H)}\frac{\abs{E(H[X])}}{\abs{X}}$ for every graph~$H$ where $H[X]$ denotes the subgraph of~$H$ induced by the vertex set~$X$.\footnote{In fact, pseudoarboricity coincides with \emph{arboricity} (that is the $\calF$-covering number for the class~$\calF$ of all forests \cite[Theorem~8]{knauer2016threeways}).}
They prove that $p(H) \leq \cl{\calS}(H) \leq p(H)+1$ \cite[Theorem~9]{knauer2016threeways}.
As $ \delta(H) \cdot \abs{X} \leq 2 \cdot \abs{E(H[X])} \leq \Delta(H) \cdot \abs{X}$ for every~$X \subseteq V(H)$ we thus obtain a lower and an upper bound on~$\cl{\calS}(H)$ in the minimum and maximum degree~$\delta(H)$ and $\Delta(H)$, from which we compute~$\cl{\calBc}(G^{(p)})$.
\begin{lemma}[{Knauer, Ueckerdt \cite{knauer2016threeways}}]
	\label{thm:star-half-deg}
	For every graph $H$, we have \[\ceil{\frac{\delta(H)}{2}} \leq \cl{\stars}(H) \leq \ceil{\frac{\Delta(H)}{2}} + 1.\]
\end{lemma}

\begin{lemma}
	\label{thm:bc-p-cover}
	If $p \geq 5$ is prime, then $\cl{\biclique}(G^{(p)}) = \frac{p+3}{2}$.
\end{lemma}
\begin{proof}
    Recall that every $\calBc$-cover of~$G^{(p)}$ is an $\calS$-cover for the class~$\calS$ of star forests as~$G^{(p)}$ is $C_4$-free (cf. \cref{thm:c4-free}).
    It thus suffices to prove $\cl{\calS}(G^{(p)}) = \frac{p+3}{2}$.

    As no prime $p \geq 5$ is even and $\Delta(G^{(p)}) = p$ by \cref{thm:c4-free}\;\eqref{itm:c4-free-c4}, we obtain $\cl{\calS}(G^{(p)}) \leq \frac{p+3}{2}$ by \cref{thm:star-half-deg}.
    
    For the lower bound, let $\varphi\colon S_1 \cupdot \dots \cupdot S_t \to G^{(p)}$ be a $k$-local $\calS$-cover for some~$k \in \N$.
    If~$p-1 \leq k$, we clearly have $k \geq \frac{p+3}{2}$ since $p \geq 5$.
    Thus, $\cl{\calS}(G^{(p)}) \geq \frac{p+3}{2}$.
    If $p-1 > k$, each vertex of~$G^{(p)}$ is the center of some distinct star~$S_i$ as $\delta(G^{(p)}) = p-1$ (cf. \cref{thm:c4-free}\;\eqref{itm:c4-free-deg}).
    That is $t \geq \abs{V(G^{(p)})}$. 
    As $\delta(G^{(p)}) = p-1$ and there exists a vertex of degree~$d$ by \cref{thm:c4-free}\;\eqref{itm:c4-free-deg}, we obtain
    \begin{align*}
    k \cdot \abs{V(G^{(p)})} &\geq \sum_{v \in V(H)} \hit{\varphi}{v} = \sum_{i = 1}^t \abs{V(S_i)} \geq \sum_{i = 1}^t (\abs{E(S_i)} + 1) \\
    &\geq \abs{E(G^{(p)})} + t > \frac{p-1}{2} \abs{V(G^{(p)})} + t.
    \end{align*}
    Since $t \geq \abs{V(G^{(p)})}$ this yields $k \geq \frac{p+3}{2}$.
\end{proof}

Thus, by \cref{lem:bc-local-gadget}, there exists a $\frac{p+5}{2}$-local $\calBc$-gadget whose size is polynomial in~$p$ for every prime~$p \geq 5$.

\begin{lemma}
	\label{thm:bc-fam}
	There are $K$-local $\calBc$-gadgets for the class~$\calBc$ of all complete bipartite graphs.
\end{lemma}
\begin{proof}
    For a prime~$p$ with~$p \geq 5$, there is a graph~$G^{(p)}$ on $\mathcal{O}(p^2)$~vertices with~$\cl{\calBc}(G^{(p)}) \leq \frac{p+3}{2}$ by \cref{thm:c4-free,thm:bc-p-cover}.
    Thus, there exists a $k$-local $\calBc$-gadget~$(G_k,e_k)$ on $\mathcal{O}(p^3)$~vertices with $k = \frac{p+5}{2}$ by \cref{lem:bc-local-gadget} that can be constructed in time~$\mathcal{O}(p^6)$.
    That is, the gadgets~$(G_k,e_k)$ with $k \in K = \set{\frac{p+5}{2} \given \text{$p$ is prime with~$p \geq 5$}}$ satisfy \eqref{itm:partial-local-family-gadget} and \eqref{itm:partial-local-family-poly-size}.
    
    It remains to verify property~\eqref{itm:partial-local-family-poly-k}.
    By Betrand's postulate there is for every integer~$n \in \N$ a prime in the interval~$[n,2n]$.
    Testing for every integer~$t \in [\ell,2\ell]$ whether~$2t-5$ is not divided by any~$t'$ with $1 < t' < 2t-5$ (i.e., whether~$2t-5$ is prime) yields an $\mathcal{O}(\ell^2)$-time algorithm for finding an integer~$k \in K$ with~$k \geq \ell$ and~$k \in \mathcal{O}(\ell)$.
    Thus, \eqref{itm:partial-local-family-poly-k} holds and the gadgets~$(G_k,e_k)$ are $K$-local $\calBc$-gadgets.
\end{proof}

As the \probpartloc{$\calBc$} is $\NP$-hard for the class~$\calBc$ of complete bipartite graphs and there are~$K$-local $\calBc$-gadgets for the class~$\calBc$, the \probloc{$\calBc$} is $\NP$-complete (cf. \cref{obs:NP-membership,thm:local-partial-reduction}).

\localBicliqueNP*

\section{Covers with finite guest classes}

In this section we study $\calG$-covering problems where the class~$\calG$ is finite.
This problem is closely related to the \emph{\probdecomp{$G$}}, which consists of deciding whether the edges of a given graph~$H$ can be decomposed into copies of~$G$, a problem studied by Holyer \cite{holyerNPCompletenessEdgePartitionProblems1981}.
They show that the \probdecomp{$K_n$} is $\NP$-hard for every~$n \geq 3$. 
As it is known that the $G$-decomposition problem can be solved in polynomial time when~$\abs{E(G)} \leq 2$, Holyer conjectured that deciding whether a graph~$H$ admits a $G$-decomposition is~$\NP$-hard if and only if~$G$ is a graph on at least three edges.
Yet, when~$G$ is disconnected, this does not necessarily hold. 
Brouwer and Wilson, and independently Alon, characterize graphs which admit $G$-decompositions when~$G$ is a matching.
Their characterization yields a polynomial time algorithm for the $G$-decomposition problem \cite{brouwerDecompositionGraphsLadder1980,alonNoteDecompositionGraphs1983}.
However, when some component of~$G$ contains at least three edges, the problem is $\NP$-hard.
\begin{restatable}[{Dor, Tarsi, Bryś and Lonc \cite{dorGraphDecompositionNPComplete1997,brysPolynomialCasesGraph2009a}}]{theorem}{holyerConj}
	\label{thm:holyer-np-p}
    Let~$G$ be a graph.
    \begin{enumerate}[label=\enumstyle{(\roman*)}, ref=\roman*]
		\item If each component of~$G$ consists of at most two edges, then the $G$-decomposition problem has a polynomial time solution.
		\item Otherwise, the problem is $\NP$-complete.
	\end{enumerate}
\end{restatable}

We make use of the $\NP$-completeness result for $G$-decompositions to show $\NP$-completeness for the \probglob{$\calG$} when~$\calG$ is finite and there is a unique edge-maximal graph, and discuss how this result might be generalized (see \cref{sec:finite_global}) and to which extent an analogue might hold in the local setting (see \cref{sec:finite_local}).
Another application of \cref{thm:holyer-np-p} yields an example of a finite, monotone graph class~$\calG$ where the \probloc{$\calG$} is $\NP$-complete, but $\cg{\calG}(H)$ can be determined in constant time for every graph~$H$ (see \cref{sec:global-easy_local-hard}).

\subsection{Global Covers}
\label{sec:finite_global}
Using a simple counting argument \cref{thm:holyer-np-p} directly yields $\NP$-completeness of the \probglob{$\calG$} when~$\calG$ is finite and there is a unique edge-maximal graph.
\begin{theorem}
	\label{thm:fin-glob-maxe-np}
    Let~$\calG$ be a finite graph class with a single graph~$G \in \calG$ with maximum number of edges in~$\calG$.
    If some component of~$G$ has at least three edges, then the \probglob{$\calG$} is $\NP$-complete.
\end{theorem}
\begin{proof}
    As $\calG$ is finite the \probglob{$\calG$} clearly lies in~$\NP$ by \cref{obs:NP-membership}.
	To show $\NP$-hardness, we reduce from the \probdecomp{$G$} which is $\NP$-complete by \cref{thm:holyer-np-p}.
    It suffices to observe that a graph~$H$ admits a $G$-decomposition if and only if~$H$ has a $k$-global $\calG$-cover where $k = \frac{\abs{E(H)}}{\abs{E(G)}}$. 
    Indeed, every $G$-decomposition is such a cover.
    Conversely, for every $k$-global $\calG$-cover~$\varphi\colon G_1 \cupdot \dots \cupdot G_k \to H$ of~$H$ we have $\abs{E(H)} \leq \sum_{i=1}^k \abs{E(G_i)} \leq k \cdot \abs{E(G)}$
    where equality only holds if each~$G_i$ is a copy of~$G$ and no two copies share an edge.
    That is, $\varphi$ is a $G$-decomposition of~$H$.
\end{proof}

Above theorem yields $\NP$-completeness of the \probglob{$\calG$} for many natural finite graph classes, such as paths, cycles, stars and complete graphs of bounded size.
Yet, can we drop the assumption that there is a unique graph $G \in \calG$ with maximum number of edges?
The condition turns out to be necessary: let~$\ebound{s}$ be the class of graphs with at most~$s$ edges and no isolated vertices.
Clearly, determining the global $\ebound{s}$-covering number~$\cg{\ebound{s}}(H)$ of~$H$ merely consists of counting the number of edges of~$H$.
\begin{proposition}
	\label{thm:all-edge-bounded-global-cover}
	For every $s \in \mathbb{N}$ and every graph~$H$, we have $\cg{\ebound{s}}(H) = \ceil{\frac{\abs{E(H)}}{s}}$.
    In particular, the \probglob{$\ebound{s}$} can be solved in constant time.
\end{proposition}

We believe that if each graph of a graph class~$\calG$ is connected, then the assumption of a unique edge maximal graph can be dropped in \cref{thm:fin-glob-maxe-np}.
\begin{restatable}{conjecture}{conjFiniteGlobalConnected}
    \label{conj:holyer-glob}
	If~$\calG$ is a finite class of connected graphs such that $\max_{G \in \calG} \abs{E(G)} \geq 3$ then the \probglob{$\calG$} is $\NP$-hard.
\end{restatable}

\subsection{Local Covers}
\label{sec:finite_local}

For the global setting, the \probglob{$\calG$} is $\NP$-complete when there is a unique edge-maximal graph~$G \in \calG$ and one component of~$G$ has at least three edges. 
Does the same hold in the local setting?
The class~$\calS_{\leq d}$ of stars of maximum degree~$d$ provides a counter-example.
\begin{theorem}
	\label{thm:fin-local-bounded-star}
    For an integer~$d \in \N$, the \probloc{$\calS_{\leq d}$} can be decided in polynomial time.
\end{theorem}
We show above theorem by adapting a proof given by Knauer and Ueckerdt, who compute local $\calS$-covering numbers for the class~$\calS$ of stars in polynomial time \cite[Theorem~25]{knauer2016threeways}.
\begin{proof}
 It suffices to show that a graph~$H$ satisfies $\cl{\calS_{\leq d}}(H) \leq k$ if and only if there exists an orientation of~$H$ such that $\frac{d \cdot (\deg(v)-k)}{d-1} \leq \indegree(v)$ for every vertex~$v \in V(H)$.
The existence of such an orientation can then be checked with a flow-algorithm in polynomial time, see \cite{frankHowOrientEdges1978}.

Orienting all edges of each guest~$S_i$ of a $k$-local $\calS_{\leq d}$-cover~$\varphi\colon S_1 \cupdot \dots \cupdot S_t \to H$ towards its center yields an orientation of~$H$ (we may assume that~$\varphi$ is a decomposition), see \cref{fig:star-local-orientation}. 
For every vertex $v \in V(H)$ we have
\begin{align*}
\label{eq:star_orientation}
    (\deg(v) - \indegree(v)) + \ceil*{\frac{\indegree(v)}{d}} \leq k
    &\iff (\deg(v) - \indegree(v)) + \frac{\indegree(v)}{d} \leq k \\
    &\iff \frac{d \cdot (\deg(v)-k)}{d-1} \leq \indegree(v) \tag{$\star$}.
\end{align*}
Conversely given an orientation of~$H$ where $\frac{d \cdot (\deg(v)-k)}{d-1} \leq \indegree(v)$ for every vertex~$v \in V(H)$, decomposing the incoming edges at every vertex into $\ceil{\frac{\indegree(v)}{d}}$ stars yields a (not necessarily unique) $k$-local $\calS_{\leq d}$-cover of~$H$ by \eqref{eq:star_orientation}.
\begin{figure}
    \centering
    \includegraphics[page=1]{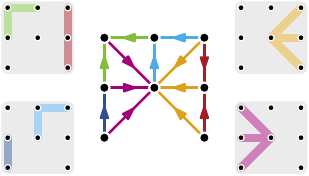}
    \caption{A $3$-local $\calS_{\leq 3}$-cover of a graph~$H$ and its corresponding orientation.}
    \label{fig:star-local-orientation}
\end{figure}
\end{proof}

While the \probloc{$\calG$} may be solvable in polynomial time even when some components of graphs in~$\calG$ contain many edges, we believe that if~$\calG$ consists of a single graph a local analogue of Holyer's conjecture may hold. 
\begin{restatable}{question}{conjFiniteHolyerLocal}
	\label{conj:holyer-local}
	Let $\calG = \set{G}$ be a graph class consisting of a single graph.
	\begin{enumerate}[label=\enumstyle{(\roman*)}, ref=\roman*]
		\item\label{itm:holyer-local-paths} Has the \probloc{$\calG$} a polynomial time solution if each component of~$G$ consists of at most two edges?
		\item Is the \probloc{$\calG$} otherwise $\NP$-complete?
	\end{enumerate}
\end{restatable}

\subsection{Computing Local Covering Numbers Can Be \texorpdfstring{$\bm{\NP}$}{NP}-Hard While Computing Global Covering Numbers Only Needs Constant Time}
\label{sec:global-easy_local-hard}
In general it seems that global covering problems are at least as hard as their respective local variant.
Indeed, this is the case for $\calG$-covering problems where $\calG$ is the class of bipartite graphs, complete bipartite graphs, cliques, interval graphs, forests and star forests respectively, see \cref{tab:overview} for an overview. 
Knauer and Ueckerdt ask whether in general there exists a graph class~$\calH$ such that the local $\calG$-covering number~$\cl{\calG}(H)$ can be computed in polynomial time for every graph~$H \in \calH$ while computing global $\calG$-covering numbers is $\NP$-hard for the class~$\calH$ \cite[Question~26]{knauer2016threeways}.
This turns out to be the case.
Stumpf provides examples of 
\begin{itemize}
    \item graph classes~$\calH$ and~$\calG$ such that it is undecidable whether~$\cl{\calG}(H) \leq k$ for $H \in \calH$ while the global $\calG$-covering numbers~$\cl{\calG}(H)$ can be computed in constant time \cite[p.\,38-40]{stumpf2015CoveringNumbersDifferent}, and
    \item graph classes~$\calH$ and~$\calG$ such that determining whether~$\cg{\calG}(H) \leq k$ for $H \in \calH$ is $\NP$-complete while~$\cg{\calG}(H)$ can be computed in constant time \cite[Section~6.2]{stumpf2015CoveringNumbersDifferent}.
\end{itemize}
In the first case, neither the guest class~$\calG$ nor the host class~$\calH$ are natural.
In the latter case, the guest class~$\calG$ they consider is the class~$\calI$ of all interval graphs, yet their construction heavily relies upon restricting the host graphs~$\calH$ (in particular to compute~$\cl{\calG}(H)$ in polynomial time):
the host class they consider consists of disjoint unions of graphs~$H \cupdot L_{k+1}$ where $L_{k+1}$ is a fixed graph with $\cl{\calG}(L_{k+1}) = 2$ and $\cg{\calG}(L_{k+1}) = k+1$.

We provide an example where both the host class~$\calH$ and guest class~$\calG$ are natural graph classes, more precisely $\calH$ consists of all graphs and~$\calG$ is monotone and finite. 

\thmGlobalEasyLocalHard*
\begin{proof}
    Note that \eqref{itm:global-easy} is a special case of \cref{thm:all-edge-bounded-global-cover}.
    To prove \eqref{itm:local-hard}, first observe that the \probloc{$\etri$} lies in~$\NP$ by \cref{obs:NP-membership}. 
    To show $\NP$-hardness, we reduce from the \probdecomp{$K_3$} which is $\NP$-hard by \cref{thm:holyer-np-p}.
    Let~$D$ be a graph of maximum degree~$\Delta(D)$ that is an instance of the decomposition problem.
    We may assume that all vertex degrees of~$D$ are even as otherwise $D$ admits no $K_3$-decomposition.
    Let~$H$ be the graph obtained from~$D$ by adding $3 \cdot \frac{(\Delta(D) - \deg_D(v))}{2}$ leaves to every vertex~$v \in V(D)$ and set $k = \frac{\Delta(D)}{2}$.
    We call the edges~$E(H)-E(D)$ \emph{new} and all other edges of~$H$ \emph{old}.
    It suffices to prove that~$D$ admits a $K_3$-decomposition if and only if~$\cl{\etri}(H) \leq k$.

    If~$D$ admits a $K_3$-decomposition, we can cover the old edges of~$H$ with triangles such that each vertex~$v \in V(H)$ is hit $\frac{\deg_D(v)}{2}$ times. 
    For each vertex~$v \in V(H)$, the remaining $3 \cdot \frac{(\Delta(D) - \deg_D(v))}{2}$ incident edges can be partitioned into $\frac{(\Delta(D) - \deg_D(v))}{2}$ stars on three edges.
    That is, we obtain an $\etri$-cover of~$H$ where every vertex of~$D$ is hit $k = \frac{\Delta(D)}{2}$ times and all other vertices once, i.e., $\cl{\etri}(H) \leq k$.

    Now suppose~$H$ admits a $k$-local $\etri$-cover $\varphi\colon G_1 \cupdot \dots \cupdot G_t \to H$. 
    We may assume that each~$G_i$ covers at least one edge of~$H$.
    In order to prove that the restriction of~$\varphi$ to~$D$ is a $K_3$-decomposition we assign weights~$w(G_i)$ and costs~$c(G_i)$ to the graphs~$G_i$ such that the larger the weight~$w(G_i)$, the larger the portion of~$D$ covered by~$G_i$, and the larger the cost~$c(G_i)$ the more of the vertices of~$D$ are hit by~$G_i$.
    Old edges~$e \in E(D)$ have weight~$w(e)=1$, new edges~$e \in E(H-D)$ a weight of~$w(e) = \frac{1}{3}$. 
    The cost~$c(G_i)$ of a graph~$G_i$ corresponds to the number of vertices of~$D$ hit by~$G_i$ and its weight~$w(G_i)$ to the sum of edge weights it covers, i.e., $c(G_i) = \abs{V(D) \cap V(\varphi(G_i))}$, and $w(G_i) = \sum_{e \in E(G_i)} w(\varphi(e))$. 
    As~$\varphi$ is a $k$-local cover of~$H$, we have
    \begin{align*}
    \label{eq:weight_cost_local_hard}
		\sum_{i \in [t]} w(G_i) &\geq \sum_{e \in E(H)} w(e) = \abs{E(D)} + \frac{1}{3}\cdot \abs{E(H) - E(D)} \\
        &= \sum_{v \in V(D)}\frac{\deg_D(v)}{2} + \sum_{v \in V(D)} \frac{\Delta(D) - \deg_D(v)}{2} 
        = k \cdot \abs{V(D)} \geq \sum_{i \in [t]} c(G_i) \tag{$\star$}.
	\end{align*}

    We show that 
    \begin{enumerate}[label=\enumstyle{(\roman*)}, ref=\roman*, start=3]
        \item\label{itm:cost-weight-equality}  $c(G_i) \geq w(G_i)$ for every~$i$ and that equality only holds when $G_i$ is a triangle of old edges or a star on three new edges.
    \end{enumerate}
    It then follows with \eqref{eq:weight_cost_local_hard} that each edge of~$H$ is covered exactly once and that $c(G_i) = w(G_i)$ for every $i$.
    In particular the old edges are decomposed into triangles, that is $\varphi$ induces a $K_3$-decomposition of~$D$.

    It remains to prove \eqref{itm:cost-weight-equality}.
    Consider a graph~$G_i$ with $i \in [t]$.
	\begin{itemize}
		\item If $c(G_i) = 0$, then $G_i$ does not cover any edge of $H$ since every edge of $H$ has an endpoint in $D$, a contradiction.
		\item If $c(G_i) = 1$, then $G_i$ contains only one vertex of $D$. Thus, $G_i$ only covers new edges, each of which has weight~$\frac{1}{3}$. Since $\abs{E(G_i)} \leq 3$ we obtain $w(G_i) \leq c(G_i)$ where equality holds if and only if $G_i$ is a star on three new edges.
		\item If $c(G_i) = 2$, then $G_i$ contains two vertices of $D$ and thus at most one edge of $E(D)$, i.e., $w(G) \leq 1 + \frac{2}{3} < 2 = c(G)$.
		\item If $c(G_i) \geq 3$, we have $w(G_i) \leq c(G_i)$ since $w(G_i) \leq 3$. Equality holds if and only if $G_i$ consists of three old edges.  In this case, $G$ is a $K_3$ since this is the only graph on three vertices with three edges.\qedhere
	\end{itemize}
\end{proof}

\section{Discussion}

Within this work, we provide a lower bound on the local $\calB$-covering number in terms of the fractional chromatic number (where $\calB$ denotes the class of all bipartite graphs).
In \cref{discussion:fractional_chromatic} we discuss open questions with regard to the relationship between these two parameters. 
Further, we show $\NP$-completeness of the \textsf{local}-$\calB$- and $\calBc$-\textsf{covering} problem (where $\calBc$ denotes the class of all complete bipartite graphs), yet the complexity of the corresponding \textsf{local}-$k$-$\calB$- and $k$-$\calBc$-\textsf{covering} problems remains open for most $k$, see \cref{discussion:strong_np} for a discussion.
Finally, we provide a finite graph class~$\calG$ such that the \probloc{$\calG$} is $\NP$-complete while the global $\calG$-covering number can be determined in constant time.
When $\calG$ is finite, the covering problem is closely related to the $G$-decomposition problem. 
We discuss how results from the field of $G$-decompositions might translate to global and local covers with finite guest classes~$\calG$ in \cref{discussion:finite_guest_class}.

\subsection{An upper bound on the local bipartite covering number?}
\label{discussion:fractional_chromatic}
We provide a logarithmic lower bound on the local $\calB$-covering number~$\cl{\calB}$ (where~$\calB$ is the class of all bipartite graphs) in terms of the fractional chromatic number~$\chif$, see \cref{thm:fractional-chromatic-cover-bip}.
Yet, can $\cl{\calB}$ also be bounded from above in terms of~$\chif$?
While the global $\calB$-covering number is tied to the chromatic number (\cref{lem:cover_with_log_chi_bipartite_graphs}), we do not believe the same holds for the local variant and the fractional chromatic number.
\begin{question}
\label{conj:local-bip-upper-bound}
    \begin{enumerate}[label=\enumstyle{(\roman*)}, ref=\roman*]
    \item\label{conj:fract-chrom-unbounded-bip} Is there a function~$f$ for the graph class~$\calB$ of all bipartite graphs such that~$\cl{\bip}(H) \leq f(\chif(H))$ for every graph~$H$?
    \item\label{conj:fract-chrom-unbounded} Is there for every graph class~$\calF_r = \{G \mid \chif(G) \leq r\}$ with~$r \in \mathbb{R}_{\geq 2}$ a function~$f_r$ such that~$\cl{\calF_r}(H) \leq f_r(\chif(H))$ for every graph~$H$?
    \end{enumerate}
\end{question}

In fact, \cref{conj:local-bip-upper-bound}\,\eqref{conj:fract-chrom-unbounded-bip} is a special case of \eqref{conj:fract-chrom-unbounded} as $\calF_2 = \calB$ since odd cycles have fractional chromatic number greater than~$2$.
Indeed, any~$\by{2b}{b}$-coloring corresponds to a proper $2$-vertex-coloring as all neighbors of each vertex receive the same color set.

In order to answer \cref{conj:local-bip-upper-bound} in the negative, it suffices to prove that there exists a graph class~$\calH$ of bounded fractional chromatic number with $\cl{\calB}(\calH) = \infty$.
Note that such a family in particular needs to satisfy $\chi(\calH) = \infty$ as $\cl{\calB}(H) \leq \cg{\calB}(H) = \log(\chi(H))$ for every graph~$H$ (cf. \cref{lem:cover_with_log_chi_bipartite_graphs}).
We believe that a subfamily of \emph{Kneser graphs} provides such an example. 
The \emph{Kneser graph}~$\kneser{a}{b}$ is the graph whose vertex set consists of all $b$-element-subsets of $[a]$ and where two vertices are connected if their subsets are disjoint.
Note that~$\kneser{a}{b}$ has no edges if $a < 2b$.
Kneser graphs appear naturally in the study of fractional colorings and play a similar role to complete graphs with respect to the chromatic number:
the Kneser graph~$\kneser{a}{b}$ with $a \geq 2b$ satisfies~$\chif(\kneser{a}{b}) = \frac{a}{b}$ (and also $\chi(\kneser{a}{b}) = a-2b+2$)  \cite[Proposition~3.2.4 and~3.2.6]{scheinermanFractionalGraphTheory2011}.
In fact, for every graph~$H$ with~$\chif(H) \leq \frac{a}{b}$ there exists a graph homomorphism $\gamma \colon H \to \kneser{a}{b}$ (yet $\gamma$ might not be injective) \cite[Proposition~3.2.4]{scheinermanFractionalGraphTheory2011}.
We believe that there exists some~$s \in \mathbb{R}$ such that the graph class~$\calH_s = \set{\kneser{a}{b} \given \frac{a}{b} \leq s}$ satisfies~$\cl{\bip}(\calH_s) = \infty$, giving a negative answer to \eqref{conj:fract-chrom-unbounded-bip} and further that \eqref{conj:fract-chrom-unbounded} can be answered in a similar way.

\subsection{Strong \texorpdfstring{$\bm{\NP}$}{NP}-completeness}
\label{discussion:strong_np}
The $k$-\probloc{$\calB$} for the class~$\calB$ of all bipartite graphs is clearly in~$\P$ for~$k=1$ as it suffices to check whether the input is a bipartite graph.
For~$k=2$, we show $\NP$-completeness (cf. \cref{thm:local-bip-np}), yet the problem remains open for all $k \geq 3$.
\begin{question}
\label{question:local_bip_NP}
    Is the $k$-\probloc{$\calB$} $\NP$-complete for every $k \geq 2$?
\end{question}
We believe this to be true.
As we have $\NP$-completeness for~$k=2$, any polynomial-size construction which increases the local $\calB$-covering number by one would provide a positive answer to \cref{question:local_bip_NP}.
The \emph{join}~$H \join H$ of a graph~$H$ with itself (that is the graph consisting of two copies~$H',H''$ of~$H$ where all vertices of~$H'$ are connected to all vertices of~$H''$) might provide such a construction.

Similarly, while we show $\NP$-completeness of the \probloc{$\calBc$} for the class~$\calBc$ of all complete bipartite graphs (cf. \cref{thm:local-bc-np}), the complexity of the $k$-\probloc{$\calBc$} remains open for every $k \geq 2$.
\begin{question}
\label{question:local_bc_NP}
    Is the $k$-\probloc{$\calBc$} $\NP$-complete for every $k \geq 2$?
\end{question}
Note that it suffices to prove $\NP$-completeness for~$k=2$ by \cref{cor:k-local-Bc-hard-then-k+1}.

\subsection{Complexity of the global and local \texorpdfstring{$\bm{\calG}$}{G}-covering problem for finite graph classes~\texorpdfstring{$\bm{\calG}$}{G}}
\label{discussion:finite_guest_class}

Holyer conjectured that the $G$-decomposition problem is $\NP$-complete if~$G$ consists of at least three edges \cite{holyerNPCompletenessEdgePartitionProblems1981}.
While the statement does not hold in its entirety, a small adaptation turns out to be true.
\holyerConj*

Similarly, if there is a unique edge-maximal graph in~$\calG$, the \probglob{$\calG$} is $\NP$-complete (cf. \cref{thm:fin-glob-maxe-np}).
Yet, the condition of a single edge-maximal graph cannot be dropped. 
Indeed, the \probglob{$\ebound{s}$} is clearly solvable in polynomial time for the graph class~$\ebound{s}$ consisting of all graphs on at most $s$~edges (cf. \cref{thm:all-edge-bounded-global-cover}).
However, many graphs in~$\ebound{s}$ are disconnected.

\conjFiniteGlobalConnected*

On the contrary, the \probloc{$\calG$} may be solvable in polynomial time even when some components of graphs in~$\calG$ contain many edges: the \probloc{$\calS_{\leq d}$} is solvable in polynomial time for the class~$\calS_{\leq d}$ of all stars with at most $d$~edges (cf. \cref{thm:fin-local-bounded-star}).
Yet, does a local analogue to Holyer's conjecture hold if~$\calG$ consists of a single graph? 
\conjFiniteHolyerLocal*
Schwebler shows this to be true for connected $r$-regular graphs~$G$ on at least three edges and for~$G$ being a star with at least three leaves \cite[Theorem~5.13 and~5.15]{schweblerGraphCoveringAlgorithnms2026}.
However, in particular the case where $G=K_{1,2}$ (i.e., $G$ corresponds to a path on three vertices) remains open, even though we know the \probloc{$\calG$} to be solvable in polynomial time for~$\calG = \set{K_2, K_{1,2}}$ by \cref{thm:fin-local-bounded-star}.

A full answer to \cref{conj:holyer-local} seems quite difficult since the proof of \cref{thm:holyer-np-p} is very technical and the ideas might not be directly applicable to local covers.
In the global setting, all $G$-decompositions are equally good since they require the same number of copies of~$G$.
However, this is not true in the local setting: it matters how often individual vertices are hit.
Thus, a modification of the ideas used to prove the $\NP$-completeness result of \cref{thm:holyer-np-p} presented in \cite{dorGraphDecompositionNPComplete1997} might require rebuilding some of the machinery.
In particular Dor and Tarsi make use of Wilson's theorem, which provides a complete characterization of graphs~$G$ such that~$K_n$ admits a $G$-decomposition for large enough~$n$.
\begin{theorem}[Wilson's theorem \cite{wilsonDecomposition1976}]
    For every graph~$G$ and large enough~$n$, the complete graph~$K_n$ admits a $G$-decomposition if and only if
	\begin{enumerate}[label=\enumstyle{(\roman*)}, ref=\roman*]
		\item $\abs{E(G)}$ divides $\abs{E(K_n)}$, and
		\item the greatest common divisor of the values~$\deg_G(v)$ with $v \in V(G)$ divides $(n-1)$.
	\end{enumerate}
\end{theorem}
It is easy to see that these two conditions are necessary, but it is highly non-trivial that they are sufficient when~$n$ is large enough.
To show $\NP$-completeness of the $G$-decomposition problem, Dor and Tarsi use Wilson's theorem to construct gadgets \cite{dorGraphDecompositionNPComplete1997}.
Since Wilson's theorem does not guarantee the existence of a $G$-decomposition hitting each vertex of $K_n$ equally often, a modification of their proof idea to  local $G$-covers might require a stronger version of Wilson's theorem.

\bibliographystyle{plainurl}
\bibliography{references}
\end{document}